\documentclass[11pt]{amsart}

\usepackage[english]{babel}
\usepackage[T1]{fontenc}
\usepackage{soul}

\usepackage{amssymb}
\usepackage{amsmath}
\usepackage{amsthm}

\newtheorem{theorem}{Theorem}[section]
\newtheorem*{theoremF*}{Theorem F}
\newtheorem*{theoremL*}{Theorem L}
\newtheorem{lemma}[theorem]{Lemma}
\newtheorem{proposition}[theorem]{Proposition}
\newtheorem{corollary}[theorem]{Corollary}

\newtheorem{fact}[theorem]{Fact}

\theoremstyle{definition}
\theoremstyle{remark} 

\newtheorem{definition}[theorem]{Definition}
\newtheorem{remark}[theorem]{Remark}

\newtheorem{notation}[theorem]{Notation}

\newcommand{\ep}{\varepsilon}
\newcommand{\vf}{\varphi}

\newcommand{\conv}{\operatorname{conv}}

\newcommand{\dist}{\mathrm{dist}}

\newcommand{\Tan}{\operatorname{Tan}}
\newcommand{\Nor}{\operatorname{Nor}}

\newcommand{\wtilde}{\widetilde}

\newcommand{\R}{\mathbb{R}}

\newcommand{\N}{\mathbb{N}}

\newcommand{\M}{\mathcal{M}}

\newcommand{\cL}{\mathcal{L}}

\newcommand{\cD}{\mathcal{D}}

\newcommand{\INt}{{\rm int}\,}
\newcommand{\interior}{{\rm int}\,}

\newcommand{\diam}{{\rm diam}\,}
\newcommand{\graph}{\operatorname{graph}}
\newcommand{\hyp}{\operatorname{hyp}}
\newcommand{\epi}{\operatorname{epi}}

\newcommand{\reach}{{\rm reach}\,}

\newcommand{\id}{{\rm id}}

\newcommand{\spa}{\operatorname{span}}
\newcommand{\Unp}{\operatorname{Unp}}

\newcommand{\B}{\operatorname{B}}

\title[Two-dimensional sets with positive reach]
{A characterization of the local structure  of two-dimensional sets with positive reach}

\author{Jan Rataj\and Lud\v ek Zaj\'\i\v cek} %% Author name
\address{Charles University, Faculty of Mathematics and Physics, Sokolovsk\'a 83, 186 75 Praha 8, Czech Republic}
\email{rataj@karlin.mff.cuni.cz}
\email{zajicek@karlin.mff.cuni.cz}

\begin{document}           

%% Abstract
\begin{abstract}
The main result of the article is a complete characterization of the local structure  of two-dimensional sets with positive reach in $\R^d$. 
We also present a more elementary proof of a recent result of A. Lytchak which describes for $k\leq d$ the local  structure  of $k$-dimensional sets with positive reach $A$ in $\R^d$ at points where the tangent cone of $A$ is $k$-dimensional. 
As an easy corollary of our and Lytchak's results we obtain a characterization of compact two-dimensional sets with positive reach in $\R^d$.
Our method also shows that, for any set $A\subset\R^d$ with positive reach, the set of points at which the tangent cone of $A$ is $k$-dimensional is locally contained in a $k$-dimensional $C^{1,1}$ surface. As a consequence we obtain that if $1\leq k<d$, and $A$ is $k$-dimensional, it can be covered by countably many $k$-dimensional $C^{1,1}$ surfaces.
\end{abstract}


%% Keywords
%\begin{keyword}
%set with positive reach \sep two-dimensional set \sep local structure \sep tangent cone \sep $C^{1,1}$ diffeomorphism
%\end{keyword}

\maketitle

\section{Introduction}

A set $A\subset\R^d$ is said to have \emph{positive reach} if for some $\ep>0$, any point $x\in\R^d$ with $\dist(x,A)<\ep$ have its unique nearest point $a\in A$ (point with $\dist(x,A)=|x-a|$). Sets with positive reach, which form a common extension of closed convex sets and compact $C^2$-domains with boundary,  were introduced  and applied by Federer \cite{Fed59}. Later their natural generalization have been investigated also in Riemannian manifolds (see \cite{Ban82}) and in general Hilbert spaces (see e.g. \cite{CT10,ThII}), under the name \emph{prox-regular sets}. Currently, sets with positive reach are extensively used in data analysis for manifold reconstruction (see e.g.\ \cite{BCM22}).

It appears that sets with positive reach may have much  more complicated geometric or topological  structure than closed convex sets or compact $C^2$-domains. Even the problem of complete
 characterization of the local structure of sets with positive reach in $\R^d$ seems to be too difficult. 
However, satisfactory answers are known in some special cases.

The following result   is a consequence of statements claimed in \cite[Remark 4.20]{Fed59}
 and easily follows from \cite[Proposition 7.4]{RZ17}.

\begin{theoremF*} [Federer]
Let  $A \subset \R^d$ be a $k$-dimensional set ($1\leq k<d$) with positive reach, $a \in A$,
and let the tangent cone of $A$ at $a$ be a $k$-dimensional space. Then $A$ agrees on a neighbourhood of $a$ with  a $k$-dimensional $C^{1,1}$ surface.
\end{theoremF*}

By well-known results, if $P$ is a $k$-dimensional $C^{1,1}$ surface (cf.\ Definition~\ref{plochy}), $a\in P$ and $r>0$ is sufficiently
 small, then $P\cap\overline{B}(a,r)$ has positive reach. It follows that  Theorem~F gives a complete satisfactory
 characterization of the local structure of $k$-dimensional sets $A$ with positive reach at a fixed point $a\in\R^d$ such that the tangent cone of $A$ at $a$ is a subspace of dimension $k$: it is
 the same as the local structure of $k$-dimensional $C^{1,1}$ surfaces containing $a$. (More formally, the two systems
 of the corresponding germs coincide.)

Recently, A. Lytchak proved an interesting result (\cite[Theorem 1.2]{Ly24}) which is essentially an extension of Theorem~F. Lytchak's result deals with sets in $d$-dimensional Riemannian manifold; its (essentially equivalent) Euclidean version is the following.

\begin{theoremL*} [Lytchak]
Let  $A \subset \R^d$ be a $k$-dimensional set ($1\leq k\leq d$) with positive reach, $a \in A$, and let the tangent cone of $A$ at $a$ have dimension $k$. Then there exist a convex body $K\subset\R^k\subset \R^d$ (of dimension $k$), an open set $U\subset\R^d$ containing $K$ and a $C^{1,1}$-diffeomorphism  $\Phi: U \to \R^d$  such that $\Phi(K)$ is a neighbourhood of $a$ in $A$.
\end{theoremL*}

Since any set  $\Phi(K)$ as above and containing $a$ has all properties of the set $A$ from Theorem~L, we see that Theorem~L gives a characterization of the local structure of sets $A$ with positive reach at points $a$ where the dimension of the tangent cone is the same as the dimension of $A$ (cf.\ our remark after Theorem~F).

Note that Theorem~L not only generalizes Theorem~F (which corresponds to the case $\Phi^{-1}(a)  \in \interior K$), but also gives (for $k=1$) a local characterization of one-dimensional sets $A$ with positive reach in $\R^d$ which is equivalent to that which follows from  \cite[Corollary~8.9]{RZ17}: on a neighbourhood of $a\in A$, $A$ is either a singleton, or a simple $C^{1,1}$ arc. 

If $A\subset\R^2$ has positive reach then its local description using semiconcave and semiconvex functions (without any further restrictions) follows easily from \cite[Theorem~6.4]{RZ17}. This result, however, cannot be easily extended to higher dimensions.

The main result of the present article is a complete characterization of the local structure of two-dimensional sets $A$ with positive reach in $\R^d$ for $d\geq 3$. Note that Theorem~L gives such a characterization at points $a\in A$ where $\Tan(A,a)$ has dimension $2$; the only nontrivial remaining case is when $\dim\Tan(A,a)=1$.  This case, however, is much  different from the case $\dim\Tan(A,a)=2$ - in particular, we cannot expect that $A$ will lie on a two-dimensional $C^{1,1}$ surface in some neighbourhood of $a$. This follows already from \cite[Example~7.13]{RZ17}, which we will describe now in a slightly generalized version.

Consider $r>0$ and Lipschitz functions
$\psi \leq 0 \leq \vf$ on $[0,r]$ such that $\vf$ is semiconcave on $(0,r)$, $\psi$
is semiconvex on $(0,r)$, $\vf(0)= \psi(0)=0$, $\vf'_+(0)= \psi'_+(0)=0$, and set
$B:= \{(x,y):\, x\in [0,r],\,  \psi(x) \leq y \leq \vf(x)\}$.
Then $B$ has positive reach (see Lemma~\ref{B_set_PR}).
We further perturbe $B\subset\R^2\subset\R^3$ to a set $M\subset \R^3$ by a ``multirotation'' by the following way.
We set $C:= \{x\in [0,r]: \vf(x)=\psi(x)\}$, denote by $\Lambda$ the set of all components of $(0,r)\setminus C$ and for each $\lambda \in \Lambda$ denote $B_\lambda = \{(x,y) \in B:\ x\in \overline \lambda\}$. Now we rotate each  $B_\lambda$  in $\R^3$ around the $x$-axis by an angle $\theta(\lambda)$ in the positive sense to a set $M_\lambda$ and set  
\begin{equation}  \label{M}
    M:= (B\setminus \bigcup _{\lambda \in \Lambda} B_\lambda)
 \cup   \bigcup _{\lambda \in \Lambda} M_\lambda.
\end{equation} 
Then  $M$ has positive reach as well (it follows from Lemma~\ref{prorot}) and its tangent cone at $0$ is a half-line. It is easy to see that if $C$ is totally disconnected and $0$ is its accumulation point, we can choose the rotations $M_\lambda$ of $B_\lambda$ so that no neighbourhood of the origin in $M$ lies on a two-dimensional $C^{1,1}$ surface.

Our main result shows (rather surprisingly) that the above example is in a sense universal (up to a $C^{1,1}$-diffeomorphism) for the case when the tangent cone of $A$ at $a$ is a half-line. In order to state our main result (Theorem~\ref{T_main} below), we need to include the ``both-sided'' version of the set $B$ above (where the tangent cone is the full line, see case (ii) in Definition~\ref{B-sets}). (We use the notion $\B$-set as abbreviation of ``basic set''.)
We also  define properly  ``multirotations'' in $\R^d$ (Definition~\ref{multrot}).

\begin{definition}[B-sets]  \label{B-sets}
\begin{enumerate}
    \item[(i)] We call $M\subset \R^2$  a \emph{$\B_-$-set} if there exist $r>0$, Lipschitz functions $\psi \leq 0 \leq \vf$ on $[0,r]$ such that $\vf$ is semiconcave on $(0,r)$, $\psi$ is semiconvex on $(0,r)$, $\vf(0)= \psi(0)=0$, $\vf'_+(0)= \psi'_+(0)=0$ and $M= \{(x,y):\ x\in [0,r],  \psi(x) \leq y \leq \vf(x)\}$. 
    \item[(ii)] We call $M\subset \R^2$  a \emph{$\B_+$-set} if there exist $r>0$, Lipschitz functions $\psi \leq 0 \leq \vf$ on $[-r,r]$ such that $\vf$ is semiconcave on $(-r,r)$, $\psi$ is semiconvex on $(-r,r)$, $\vf(0)= \psi(0)=0$, $\vf'(0)= \psi'(0)=0$ and $M= \{(x,y):\ x\in [-r,r],  \psi(x) \leq y \leq \vf(x)\}$.  
    \item[(iii)] We say that $M\subset\R^2$ is a \emph{$\B$-set} if $M$ is a $\B_-$-set or a $\B_+$-set.
\end{enumerate}
\end{definition}

Any $\B$-set has positive reach (see Lemma~\ref{B_set_PR}).

Let $SO(d)$ be the group of rotations in $\R^d$ (i.e., linear mappings preserving the scalar product and orientation). We call an element $R\in SO(d)$ a \emph{2-rotation}
if there exists a two-dimensional subspace $T$ of $\R^d$ orthogonal to $e_1$ and such that $Rv=v$ for any $v\in T^\perp$. (Notice that 2-rotations are determined by rotations in some 2-plane orthogonal to $e_1$.) Note that $R^{-1}$ is a 2-rotation whenever $R$ is.

Let $\pi$ denote the orthogonal projection onto the $x_1$-axis $\spa\{e_1\}$. 

\begin{definition}[multirotations]  \label{multrot}
Let $\emptyset\neq C\subset\spa\{e_1\}\subset\R^d$ ($d\geq 3$) be a compact set and $\Lambda$ the family of all components (open intervals) of $\spa\{e_1\}\setminus C$.
We will call a mapping $\rho:\R^d\to\R^d$ \emph{multirotation} (in $\R^d$) \emph{associated with $C$} if 
$\rho(x)=x$ whenever $\pi(x)\in C$ and for any $\lambda\in\Lambda$ there exists a 2-rotation $R_\lambda$ such that $\rho|_{\pi^{-1}(\lambda)}=R_\lambda|_{\pi^{-1}(\lambda)}$.
We will say that the multirotation $\rho$ is determined by the family of $2$-rotations $(R_\lambda)_{\lambda\in\Lambda}$.
\end{definition}

Given a set $A\subset\R^d$ with positive reach and $k\in\{0,1,\dots,d\}$, we denote by $T_k(A)$ the set of all points of $A$ at which the tangent cone of $A$ has dimension $k$.

Note that if $B$ is a B-set then, since it has positive reach, $T_1(B)$ is a compact subset of $\R=\spa\{e_1\}$ by Lemma~\ref{T_k}~(iii) (cf.\ also Remark~\ref{Rem_listky}~(b)).

\begin{theorem}  \label{T_main}
    Let $A\subset\R^d$ ($d\geq 3$) be a set with positive reach with $\dim A\leq 2$ and let $a\in A$ be such that $\dim(\Tan(A,a))=1$. Then there exist a $\B$-set $B\subset\R^2\subset\R^d$, a multirotation $\rho:\R^d\to\R^d$ associated with $T_1(B)$ and a $C^{1,1}$-diffeomorphism $\Phi:U\subset\R^d\to\R^d$ such that $0\in U$, $a=\Phi(0)$, $\rho(B)\subset U$ and $\Phi(\rho(B))$ is a neighbourhood of $a$ in $A$.
    
    On the other hand, if $\Phi$, $B$ and $\rho$ are as above, then $\Phi(\rho(B))$ agrees on a neighbourhood of $a$ with some set $A\subset\R^d$ with positive reach satisfying $\dim A\leq 2$ and $\dim(\Tan(A,a))=1$. 
\end{theorem}

\begin{remark}
   It is easy to see that Theorem~\ref{T_main} holds if we replace ``$\dim A\leq 2$'' with ``$\dim A=2$'' at both occurrences. The only difference is in the proof of the second part: While in the first case we set in our proof  $A:=\Phi(\rho(B))$, in the latter case we could take $A:=\Phi(\rho(B))\cup D$, where $\reach D>0$, $\dim D=2$ and $\dist(\Phi(\rho(B)),D)>0$. 
\end{remark}

Note that (ignoring the trivial case when $a$ is an isolated point of $A$), Theorem~\ref{T_main} together with Theorem~L give a complete characterization of the local structure of two-dimensional sets $A\subset\R^d$ with positive reach: For a given point $a\in\R^d$, the set of all germs at $a$ of sets $A$ as above with $a\in A$ is described ``constructively''. Also, both results together imply (using also \eqref{reach_compact} and Lemma~\ref{local}) the following characterization of compact two-dimensional sets with positive reach.

\begin{corollary}
    Let $A\subset\R^d$ ($d\leq 3$) be a two-dimensional compact set. Then $A$ has positive reach if and only if for each non-isolated $a\in A$ there exist a compact set $K\subset\R^d$, an open set $K\subset U\subset\R^d$ and a $C^{1,1}$-diffeomorphism $\Phi:U\to\R^d$ such that $\Phi(K)$ is a neighbourhood of $a$ in $A$ and either 
    \begin{enumerate}
        \item[{\rm (a)}] $K\subset\R^2\subset\R^d$ is convex two-dimensional, or
        \item[{\rm (b)}] $K=\rho(B)$, where $B\subset\R^2\subset\R^d$ is a $\B$-set and $\rho:\R^d\to\R^d$ is a multirotation associated with $T_1(B)$.
    \end{enumerate} 
\end{corollary}

In the situation of Theorem~\ref{T_main} (Theorem~L), we obtain rather easily that the mapping 
$\Psi:=\Phi\circ\rho|_B$ ($\Psi:=\Phi|_K$, respectively) is bi-Lipschitz. Using this, we obtain a positive answer to \cite[Question~1.7]{Ly24}:
 
\begin{corollary}  \label{bilip}
    Let $A\subset\R^d$ ($d\geq 3$) be a two-dimensional set with positive reach and $a\in A$. Then there exists a compact set $0\in B\subset\R^2$ with positive reach and a bi-Lipschitz mapping $\Psi:B\to A^*$, where $A^*$ is a neighbourhood of $a$ in $A$.
\end{corollary}

The proof of Theorem~\ref{T_main} has two principal ingredients which come from Lytchak's proof of Theorem~L.

First, given a set $A\subset\R^d$ with positive reach and $1\leq k< d$, we consider the mapping $\psi_k=\psi_k^A:x\mapsto\spa\{\Tan(A,x)\}$ from $T_k(A)$ into the Grassmannian $G(d,k)$ and prove its Lipschitzness on some subsets of $T_k(A)$ under additional assumptions on $A$. Second, using the Lipschitzness of $\psi_2$ on a suitable set, we construct a $C^{1,1}$-diffeomorphism applying Whitney's $C^{1,1}$ extension theorem to a suitably chosen mapping $f:D\subset\R^d\to\R^d$. Here our approach differs from that of Lytchak who also uses the Whitney's $C^{1,1}$  extension, but for a mapping from $\R^k$ to $\R^{d-k}$. For our construction of the mapping $f$ above, we need some basic facts from the theory of product-integration from \cite{GJ90}, which is the third principal ingredient of our proof.

When proving the Lipschitzness of $\psi_k$, in contrast to Lytchak who applies the theory of CAT($\kappa$) spaces, 
we use a more elementary method based on ``tangential regularity'' of sets with positive reach and some well-known facts on ``fullness'' (very close to the more frequently used ``thickness'') of simplices taken from Whitney \cite{Whi57}. 
Some additional arguments are needed to prove Lemma~\ref{lipna2} about Lipschitzness of $\psi_2$ for special two-dimensional sets  which is important in the proof of Theorem~\ref{T_main}.

Using our method we also provide an alternative (more elementary) proof of Theorem~L. 

Moreover, our method also yields the local Lipschitzness of $\psi_k$ on $T_k(A)$ for \emph{any} set $A\subset\R^d$ with positive reach, which implies that $T_k(A)$ lies locally on a $k$-dimensional $C^{1,1}$-surface (Theorem~\ref{tdmnapl2}). (Note that for $A$ $k$-dimensional, this follows already from Theorem~L.)
In the case $k=1$ we obtain (using a different method) a stronger version (Theorem~\ref{T1+}) which will be used substantially in the proof of our main result (Theorem~\ref{T_main}).
Theorem~\ref{tdmnapl2} easily implies that for $1\leq k<d$, each $k$-dimensional subset of $\R^d$ with positive reach can be covered by countably many $k$-dimensional $C^{1,1}$ surfaces (see Corollary~\ref{count}), which is a new result up to our knowledge.

The structure of the article is the following.
In Section~2 (Preliminaries), we fix some basic notation and recall some (mostly well-known) facts about $C^{1,1}$ mappings and sets with positive reach.
Section~3 is devoted to the study of Lipschitzness of the mapping
$\psi_k$. 
In Section~4 we prove results on the local structure of $T_k(A)$ (Theorems~\ref{tdmnapl2} and \ref{T1+}) and, as a consequence, Corollary~\ref{count}.
In Section~5 we prove our main results, Theorem~\ref{T_main} and Corollary~\ref{bilip}.  
Finally, in Section~6 we present an alternative proof of Theorem~L.

\section{Preliminaries}

\subsection{Basic definitions}

If $X$ is a metric space and $A\subset X$, the closure, the interior and the boundary of $A$ are denoted
by $\overline A$,  $\interior A$ and $\partial A$, respectively. 
The symbols $B(c,r)$ and $\overline B(c,r)$ denote open and closed ball of center $c$ and radius $r$,
respectively. For $A \subset X$, we consider the distance function
 $d_A(x)= \dist(x,A):= \inf\{|x-a|:\ a \in A\},\  x \in X$, and the (multivalued) metric projection 
$P_A(x)= \{a \in A:\ |x-a|= d_A(x)\}$, $x\in X.$ The open $r$-neigbourhood ($r>0$) of $A\subset X$ is defined as $ B(A,r)= \{x\in X: d_A(x) < r\}.$

The scalar product of vectors $x,y\in\R^d$ is denoted by $\langle x, y \rangle$, and $|x|=\sqrt{\langle x,x\rangle}$ is the corresponding norm.
We write $S^{d-1}$ for the unit sphere in $\R^d$.
 By $\spa\, M$  and  $\conv M$ we denote the linear span  and the convex hull of the set $M\subset \R^d$.  Under a \emph{convex body} in $\R^d$ we understand a compact convex subset with nonempty interior.
The symbol $[x,y]$ denotes the (closed) segment if $x,y\in\R^d$. 
 If $u\in \R^d$, we set  $u^{\perp}:= \spa\, \{u\}^{\perp}$. If $V$ is a linear subspace of $\R^d$, we denote
 by $\pi_V$ the orthogonal projection on $V$.
  The symbol $\mathcal H^k $ stands for the $k$-dimensional
 Hausdorff measure. For sets $A\subset \R^d$, we denote by $\dim A$ and $\dim_H A$
 the topological and Hausdorff dimensions of $A$, respectively. We use the notation $e_i$ for the $i$th canonical basis vector in $\R^d$, $i=1,\ldots,d$.

By $\cL(\R^m,\R^n)$ we mean the space of linear mappings from $\R^m$ to $\R^n$ with the usual (operator) norm $\|\cdot\|$.

We use the notation $G(d,k)$ ($0\leq k\leq d$) for the set of all $k$-dimensional linear subspaces of $\R^d$ (Grassmannian). When writing $\R^k\subset\R^d$ ($k<d$), we identify $\R^k$ with $\spa\{e_1,\dots,e_k\}$.

%If $f$ is defined on an open subset of a normed linear space $X$, we use the notation $f'_+(x,v)$ for the one sided directional derivative of $f$ at $x$ in direction $v$. 

A mapping is called $K$-Lipschitz if it is Lipschitz with a (not necessarily minimal) constant $K$. 

If $u$ is a unit vector in $\R^d$ and $f$ a continuous function defined on the space $u^\perp$, we define the \emph{$u$-hypograph} and
\emph{$u$-epigraph} of $f$ as
\begin{align*}
\hyp_uf:=&\{w+tu:\, w\in u^\perp,\, t\leq f(w)\},\\
\epi_uf:=&\{w+tu:\, w\in u^\perp,\, t\geq f(w)\}.
\end{align*}
If $u=e_d$ we write simply $\hyp f$, $\epi f$ instead of $\hyp_{e_d}f$, $\epi_{e_d}f$.

\subsection{Functions, mappings and surfaces}
\label{subs_semiconv}
A mapping $F$ between Banach spaces $X$, $Y$ is called to be $C^{1,1}$ smooth if it has Lipschitz Fr\' echet derivative $DF$ on its domain.
A mapping  $F: U\subset \R^d \to \R^d$ is called a \emph{$C^{1,1}$-diffeomorphism}
if $F$ is injective, both $U$ and $F(U)$ are open and both $F$ and $F^{-1}$ are $C^{1,1}$ smooth.

\begin{remark}\label{stability}
\begin{enumerate}
    \item[(a)] The mean value theorem yields that each $C^1$ mapping is locally Lipschitz.
    \item[(b)] It is easy to see that  $C^{1,1}$ smooth mappings (with a common domain) are stable w.r.t.\ linear combinations.
    \item[(c)] Locally $C^{1,1}$ smooth functions (resp.\ mappings) are stable w.r.t.\ multiplications (resp.\ compositions); it follows e.g.\ from \cite[Propositions 128, 129]{HJ} and (a).
\end{enumerate}
\end{remark}

We will also use the following result (on a local $C^{1,1}$-diffeomorphism).

\begin{fact}\label{locdif}
Suppose that $G\subset \R^d$ is open, $F:G \to \R^d$ is $C^{1,1}$ smooth, $a\in G$ and $DF(a): \R^d\to \R^d$
 is bijective. Then there exists an open neighbourhood $U$ of $a$ such that $F|_U$ is a $C^{1,1}$-diffeomorphism.
\end{fact}

This result is a rather easy consequence of the classical $C^1$ local diffeomorpism theorem and
 is covered e.g. in  \cite[Theorem A.9]{Hor76}. It is also  an almost immediate consequence of 
 \cite[Inverse Function Theorem]{Wells} (the case $f\in B^1_1$).

Further, we will use the following well-known
 $C^{1,1}$ version of the Whitney's extension theorem.

\begin{fact}\label{Wh}
	Let $\emptyset \neq D \subset \R^m$ be an (arbitrary) set, $f:D \to \R^n$ and $c>0$. Let 
	$\vf: D \to \cL(\R^m,\R^n)$ be a mapping such that, for each $x,y \in D$,  the following inequalities hold:
	\begin{equation}\label{Whitney_1}
	 \|\vf(y) - \vf(x)\| \leq c |y-x|,
	\end{equation}
	\begin{equation}\label{Whitney_2}
	|f(y)-f(x) - \vf(x)(y-x)| \leq c |y-x|^2.
	\end{equation}
	Then there exists a $C^{1,1}$ smooth mapping $F: \R^m \to \R^n$ which extends $f$
	 and  fulfills $F'(x)= \vf(x)$, $x\in D$.
\end{fact}
Conditions \eqref{Whitney_1} and \eqref{Whitney_2} are (first order) Whitney-Glaeser conditions
 of \cite{Gl}. The real-valued case ($n=1$) coincides with \cite[Theorem 2]{JKZ} for $X:=\R^m$
 and $\omega(t)=t,\ t>0$. The case of general $n$ follows easily by a standard way from the case $n=1$.

We will use also the following extension result.

\begin{fact}\label{Wh2}
Let $X,Y$ be finite dimensional Hilbert spaces, $U \subset X$ be an open convex set and $f: U \to Y$ be $C^{1,1}$ smooth. Then $f$ can
 be extended to a $C^{1,1}$ smooth $F: X\to Y$.
\end{fact}

The case $Y=\R$ is an immediate consequence of  \cite[Corollary 47]{JKZ} (applied with 
$\omega(t)=t,\ t>0$). For the case $Y=\R^n$ (which is clearly equivalent with the general case) we apply this result coordinate-wisely.

\begin{definition}\label{plochy}  
Let $0<k<d$.	 We say that $A \subset \R^d$ is a 
\begin{enumerate}
\item[(i)] \emph{$k$-dimensional $C^{1,1}$ g-surface} if there exist a $k$-dimensional subspace $W$ of $\R^d$ and a $C^{1,1}$  mapping $\vf: W \to W^{\perp}$ such that $A = \{w+\vf(w): \ w \in W\}$,
\item[(ii)] \emph{$k$-dimensional $C^{1,1}$ surface} if for any $a\in A$ there exist $\ep>0$ and a $k$-dimensional $C^{1,1}$ g-surface $S\subset\R^d$ such that $A\cap B(a,\ep)=S\cap B(a,\ep)$.	
\end{enumerate}
\end{definition}

\begin{remark}
$C^{1,1}$ g-surfaces are special types of $C^{1,1}$ surfaces (i.e., $C^{1,1}$ submanifolds of $\R^d$) which can be represented as graphs of $C^{1,1}$ mappings. In \cite{RZ17}, the term ``surface'' was used for ``g-surface''. All our results involving $C^{1,1}$ surfaces (Theorem~\ref{tdmnapl2}, Corollary~\ref{count} and Theorem~\ref{T1+}) could be equivalently formulated with $C^{1,1}$ g-surfaces.
\end{remark}

 One of several natural equivalent definitions  (cf. \cite[Definition 1.1.1 and Proposition 1.1.3]{CS}) of semiconcavity  reads as follows. We formulate it in the generality we need.
 
 \begin{definition}\label{semieuk} \rm
A real function $u$ on an open convex subset $C$ of a finite-dimen\-sio\-nal Hilbert space $X$  is called {\it semiconcave}  if there exists $c\geq 0$ such that
  the function $g(x) = u(x) - (c/2) |x|^2$ is  concave on $C$.
	A real function 	$v$ on $C$ is called {\it semiconvex}  if $-v$ is semiconcave.
\end{definition}

Note that if   $C$ is as 
in the above definition and $u$ is a function on $C$, then (see, e.g., \cite[Proposition 2.1.2 and Corollary 3.3.8]{CS})
\begin{equation}\label{sscjj}
\text{$u$ is $C^{1,1}$ if and only if $u$ is both semiconcave and semiconvex.}
\end{equation}

\begin{remark}  \label{rem_semicon}
If $f$ is semiconcave (semiconvex) on $(-v,v)$ then $\tilde{f}: t\mapsto f(-t)$ is semiconcave (semiconvex, resp.) as well. This follows directly from the definition.
\end{remark}

\begin{remark}  \label{RemFu}
    We will use results from \cite{Fu85} where a slightly different definition of semiconcavity is used: a function $f$ defined on an open convex set $C\subset\R^d$ is semiconcave in the sense of \cite{Fu85} if and only if it is \emph{locally} semiconcave in our sense, and it is semiconcave in our sense if and only if $\operatorname{sc}(f,C)<\infty$, see \cite[Definitions~1.3]{Fu85}.
\end{remark}

Slightly reformulating the standard definition (see e.g. \cite[Definition 1.30]{ThI}), we say that
a set-valued mapping (or multimapping) $F:x\mapsto F(x)\subset \R^d$ defined on a set $C\subset\R^d$ is \emph{lower semicontinuous} (or inner semicontinuous) if for any $x_0\in C$, $v_0\in F(x_0)$ and 
$\ep>0$ there exists $\delta>0$ such that for any $x\in C\cap B(x_0,\delta)$ we have $F(x)\cap B(v_0,\ep)\neq\emptyset$. For convex-valued multimappings, the following result holds (see \cite[Theorem~5.9~(a)]{RW04}):

\begin{lemma}  \label{conv_val}
    Let $F$ be a lower semicontinuous set-valued mapping defined on $C\subset\R^d$ with \emph{convex} values $F(x)\subset\R^d$, $x\in C$. Then for any $x_0\in C$ and $v_0\in\INt F(x_0)$ there exist $\ep,\delta>0$ such that 
    $$x\in C\cap B(x_0,\delta)\implies B(v_0,\ep)\subset F(x).$$
\end{lemma}
        
\subsection{Basic and auxiliary results on sets of positive reach}\label{bapr}
Given a set $A\subset\R^d$, we denote (following \cite{Fed59}) by $\Unp A$ the set of all points 
 $z\in\R^d$ for which the metric projection
$P_A(z):=\{ a\in A:\, \dist(z,A)=|z-a|\}$
is a singleton. 
%We will work mainly with the corresponding mapping $\Pi_A: \Unp A \to A$ (such that $P_A(z)= \{\Pi_A(z)\},\ z \in \Unp A$. 

If $A\subset \R^d$ and $a\in A$, we define (with $B(a,0):= \emptyset$)
$$\reach(A,a):=\sup\{r\geq 0:\, B(a,r)\subset \Unp A\},$$
and  
$$\reach A:=\inf_{a\in A}\reach (A,a).$$
Obviously, if $\reach A >0$, then $A$ is closed. 
Note that for $A\subset\R^d$ closed, $\reach A$ is the supremum of all $r\geq 0$ such that $x\in\Unp A$ whenever $\dist(x,A)\leq r$. Further, it is well-known that
 $\reach A = \infty$ if and only if $A$ is closed convex (cf. \cite[Remark 4.2]{Fed59}).

\begin{remark}   \label{PR-var}
    The notion of positive reach in $\R^d$ is equivalent to that of \emph{uniform prox-regularity} (i.e.,
	$r$-prox-regularity for some $r>0$)	from \cite{ThII}, and for closed sets, also to that of \emph{weak convexity} from \cite{Vial}. Consequently, we can use results from \cite{Vial} and  \cite{ThII}.

	More precisely, \cite[Theorem 15.28, (a), (y)]{ThII} shows that $A\subset \R^d$ is
	 $r$-prox regular if and only if $\reach A \geq r$. 
	
	Further, a closed set $A\subset \R^d$ is weakly convex w.r.t. $0<r< \infty$ in the sense of \cite{Vial}
	 (i.e., has Vial property with constant $r$ in the terminology of \cite[Definition 16.13]{ThII}) if and
	 only if $A$ is $r$-prox regular (i.e. $\reach A \geq r$), see  \cite[Theorem 16.15]{ThII}. 
	 
\end{remark}

If $A\subset\R^d$ and $a\in A$, we denote by $\Tan(A,a)$ the set of all tangent vectors to $A$ at $a$ 
(i.e., $u\in\Tan(A,a)$ if and only if $u=0$ or there exist $a\neq a_i\in A$ and $r_i>0$ such that $a_i\to a$ and $r_i(a_i-a)\to u$, $i\to\infty$) 
which is clearly a closed cone. 

By $\Tan_C(A,x)$ we denote the \emph{Clarke tangent cone} of $A$ at $x\in A$, which is always contained in
 $\Tan(A,x)$. The set $A$ is called \emph{tangentially regular at $x$} if $\Tan_C(A,x)=\Tan(A,x)$.

A vector $v \in \R^d$ belongs to $\Tan_C(A,x)$ (see \cite[Definition 6.25]{RW04}, where elements
of  $\Tan_C(A,x)$ are called ``regular tangent vectors''; cf.\ also \cite[Theorem 2.2]{ThI}) if $v$ has the following property:

\begin{align}\label{Clarke}
\text{ for all }x_i\in A, x_i\to x \text{ and }\tau_i\searrow 0
 &\text{ there exist points }  \tilde  x_i\in A\\ &\text{such that }  \frac{\tilde  x_i - x_i}{\tau_i} \to v.   \nonumber
\end{align}

The normal cone of $A$ at $a\in A$ is defined as the dual cone
\begin{equation}   \label{norcone}
    \Nor (A,a):=\{u\in\R^d:\, \langle u,v\rangle \leq 0\text{ for any }v\in\Tan(A,a)\}.
\end{equation}

In the following lemma we recall some well-known facts on sets with positive reach.

\begin{lemma}  \label{P_PR}
Assume that $A\subset\R^d$, $\reach A>0$ and $a\in A$.
\begin{enumerate}
\item[{\rm (i)}] The function $x\mapsto\reach(A,x)$ is either identically equal to $\infty$, or finite and  $1$-Lipschitz on $A$.
\item[{\rm (ii)}] The tangent cone $\Tan(A,a)$ is convex.
\item[{\rm (iii)}] $\Tan(A,a)\neq\R^d$ if and only if $a\in\partial A$.
\item[{\rm (iv)}] 
$\dim A=\dim_HA$ and $\dim(\Tan(A,a))\leq\dim A$. 
\item[{\rm (v)}] If $\diam A<\,\reach A$ then $A$ is connected.
\item[{\rm (vi)}] $A$ is tangentially regular at $a$.
\item[{\rm (vii)}] The multivalued mapping $x\mapsto\Tan(A,x)$ is lower semicontinuous on $A$.
\end{enumerate}
\end{lemma}

\begin{proof}
For (i) and (ii) see \cite[Proposition 3.1]{RZ17}. Assertion (iii) follows e.g.\ from \cite[Proposition 3.1~(vi)]{RZ17} and \eqref{norcone}. Statement (iv) was proved by Federer \cite{Fed59}, see \cite[(3.2), (3.3)]{RZ17}.  Assertion (v) follows from \cite[Remark~4.15~(1)]{Fed59} (cf.\ also \cite[Lemma~3.4~(ii)]{RZ17}). For (vi), see \cite[Proposition~3.8]{Vial} or \cite[Proposition 15.13 (b)]{ThII} (together with Remark~\ref{PR-var}), and (vii) follows from (vi) using \cite[Corollary~6.29]{RW04}.
\end{proof}

Note that it follows easily from Lemma~\ref{P_PR}~(i) that for a compact set $K\subset\R^d$ we have
\begin{equation}  \label{reach_compact}
    \reach(K,a)>0\text{ for all }a\in K\iff\reach K>0.
\end{equation}

We will also use  the following Federer's characterization of sets with positive reach (see \cite[Theorem~4.18]{Fed59}). 
\begin{proposition}\label{fedtan}
If $A\subset \R^d$ is a closed set and $0<r<\infty$, then the following two conditions are equivalent:
\begin{enumerate}
\item [{\rm (i)}]
$\reach A \geq r.$
\item[{\rm (ii)}]
$\dist(b-a,\Tan(A,a)) \leq |b-a|^2/(2r)$\ \ whenever\ \ $a, \,b\in A$. 
\end{enumerate}
\end{proposition}

\begin{corollary}\label{uhlovy}
Let  $A\subset \R^d$,  $0<r<\reach A$, $a,b \in A$ and  $0< |b-a|<r$. Then there exists $0\neq v\in\Tan(A,a)$ such that
\begin{equation}\label{feuh}
	 \angle(b-a,v)<\frac{\pi}{6}\cdot \frac{|b-a|}{r}.
\end{equation}
If, moreover, $\Tan(A,a)$ is one-dimensional then, denoting by $\Pi$ the orthogonal projection onto $a+\spa(\Tan(A,a))$, we have $\Pi(b)-a\in\Tan(A,a)$ and
\begin{equation} \label{tan_proj}
|\Pi(b)-a|\geq\frac{\sqrt{3}}{2}|b-a|\geq\sqrt{3}|b-\Pi(b)|.
\end{equation}
\end{corollary}

\begin{proof}
    Since $a+\Tan(A,a)$ is a closed set, we can choose $c\in P_{a+\Tan(A,a)}(b)$. 
    We have from Proposition~\ref{fedtan}
    \begin{equation} \label{fedeq}
    |b-c|\leq\frac{|b-a|^2}{2r}\leq\frac 12|b-a|.
    \end{equation}
    It follows that $c\neq a$. Clearly also $c\in P_{\{a+tv:\, t\geq 0\}}(b)$, where $v:=c-a\in\Tan(A,a)$,  and, as $c\neq a$, $c$ is also the orthogonal projection of $b$ onto the line $\{a+tv:\, t\in\R\}$ (and so $c=\Pi(b)$ if $\dim\Tan(A,a)=1$), hence $b-c\perp v$. Both \eqref{feuh} and \eqref{tan_proj} are trivial if $b=c$ and follow easily from \eqref{fedeq} using the right triangle with vertices $a,b,c$ if $b\neq c$.
\end{proof}

\begin{remark}  \label{rem_reach}
If $A\subset\R^k\subset\R^d$ ($1\leq k< d$) then $A$ has positive reach in $\R^k$ if and only if it has positive reach in $\R^d$. This well-known fact follows directly from Proposition~\ref{fedtan} (since both $b-a$ and $\Tan(A,a)$ lie in $\R^k$ if $a,b\in A$).
\end{remark}

We will need also  the following essentially well-known fact.

\begin{lemma}\label{local}
For $A \subset \R^d$ and $a \in A$, the following are equivalent.
\begin{enumerate}
\item[{\rm (i)}]  $  \reach(A,a) > 0.$
\item[{\rm (ii)}]  There exists $\delta>0$ such that  $A \cap \overline{B}(a,\delta)$ has positive reach.
\item [{\rm (iii)}] 
There exists a  set  $C$ and $\omega>0$ such that  $\reach C>0$ and $A \cap B(a, \omega)= C \cap B(a, \omega)$.
\end{enumerate}
\end{lemma}

\begin{proof}
 The
implications (ii)$\implies$(iii)$\implies$(i) are obvious. 

The non-trivial implication (i)$\implies$(ii)
 is an immediate consequence  of \cite[Lemma~3.4]{RZ17} (i) (which follows from a more general 
 \cite [Lemma 4.3]{Rat02}) since (i) implies  (by Lemma~\ref{P_PR}~(i)) that we can find $\delta>0$ such that
 $\reach(A,x)>\delta$ for all $x\in A\cap\overline{B}(a,\delta)$. (This implication also follows easily  from
 Federer's \cite[Theorem 4.10 (5)]{Fed59}, applied to $A$, $\delta$ as above  and $B:= \overline{B}(a,\delta)$.) 
\end{proof}

We will use the following fact which is an almost immediate consequence of two results of \cite{ILI}.

\begin{lemma}\label{krivky}
If $A \subset \R^d$, $d\geq 2$, $0<r<\reach A$,  $a,b\in A$ and  $0<|b-a| < 2r$, then there exists a curve $\vf: [0, l] \to  A$ parametrized by arclength such that  $\vf(0)=a$, $\vf(l)=b$,  $\vf$ is $C^{1,1}$ smooth and $l \leq\frac{\pi}{2}|b-a|.$
\end{lemma}

\begin{proof}
First note that $A$ is $r$-proximally smooth (in the sense of \cite[p.~2]{ILI}) e.g.\ by \cite[Theorem~4.8 (5)]{Fed59}. Thus \cite[Theorem~1.1]{ILI} gives that the shortest arc $\Gamma$ in $A$ between $a$ and $b$ exists, is unique, has  length  $l \leq 2r\arcsin\frac{|b-a|}{2r}$ and 
 $\Gamma$ is  an $r$-proximally smooth set. Consequently we can use \cite[Theorem~1.3]{ILI} and obtain
 that the arclength parametrization $\vf$ of $\Gamma$ with $\vf(0)=a$ and $\vf(l)=b$ is  $C^{1,1}$ smooth.
 Since $\arcsin z \leq \pi/2\cdot z,\ 0\leq z \leq 1$, we obtain  $l \leq\frac{\pi}{2}|b-a|.$   
\end{proof}

The intersection of two sets with positive reach need not have positive reach in general, unless additional assumptions are made, as in the following lemma which is a version of \cite[Proposition~3.9]{Vial}. 
%or \cite[Theorem~4.10]{Fed59}.

\begin{lemma}   \label{prunik}
    Let $A,B\subset\R^d$ be closed sets.
    \begin{enumerate}
        \item[{\rm (i)}] Let $x\in A\cap B$ be a point such that $\reach(A,x)>0$, $\reach(B,x)>0$ and 
    \begin{equation}   \label{tan_prunik}
        \Tan(A,x)\cap\INt\Tan(B,x)\neq\emptyset.
    \end{equation}
    Then $\reach(A\cap B,x)>0$.
        \item[{\rm (ii)}] If $A\cap B$ is compact, $\reach(A,x)>0$, $\reach(B,x)>0$ for any $x\in A\cap B$
     and \eqref{tan_prunik} holds for all $x\in\partial A\cap\partial B$, then $\reach (A\cap B)>0$.
    \end{enumerate}
\end{lemma}

\begin{proof}
    First we prove (i). Using Lemma~\ref{local} we see that we can assume that $\reach A>0$ and $\reach B>0$. Applying Lemma~\ref{P_PR}~(i), (ii), (vii) and Lemma~\ref{conv_val}, we find $r,\delta>0$ such that for all $y\in C:=A\cap B\cap\overline{B}(x,\delta)$, $\reach(A,y)>r$, $\reach(B,y)>r$ and \eqref{tan_prunik} holds with $y$ in place of $x$. This implies that for any $y\in C$ there do not exist nonzero normal vectors $u\in\Nor(A,y)$, $v\in\Nor(B,y)$ with $u+v=0$ (otherwise, $\Tan(A,y)\cap\Tan(B,y)$ would be contained in the hyperplane $u^\perp$, which would contradict \eqref{tan_prunik}). Thus we can apply \cite[Theorem~4.10]{Fed59}~(4) with $\rho>0$ sufficiently small (using that $\eta>0$ by \cite[Theorem~4.10]{Fed59}~(1)) and obtain (i). 
    
    For (ii), it is enough to observe that \eqref{tan_prunik} is fulfilled trivially if $x\in\INt A\cup\INt B$, and apply \eqref{reach_compact}.
\end{proof}

If we intersect a set with positive reach with a sufficiently small closed ball, the reach even does not decrease.

\begin{lemma}  \label{ball}
    Assume that $A\subset\R^d$, $\reach A\geq r>0$ and $B$ is a closed ball with radius $\rho\leq r$. Then either $A\cap B$ is empty, or $\reach(A\cap B)\geq r$.
\end{lemma}

\begin{proof}
    The result follows from \cite[Corollary~16.17]{ThII}, Remark~\ref{PR-var} and the fact that $B$ is $\rho$-strongly convex (see \cite[Definition~16.4]{ThII}), hence also $r$-strongly convex (cf.\ \cite[Corollary~16.12]{ThII}).
\end{proof}

We will use the following well-known fact (see e.g. \cite[Proposition 16.20]{ThII} together with Remark~\ref{PR-var}).

\begin{lemma}\label{str}
Let  $A \subset \R^d$ be a nonempty closed set and  $0<r< \infty$. Then the following assertions are equivalent:
\begin{enumerate}
\item[{\rm (i)}]
 $\reach A \geq r$.
\item[{\rm (ii)}]
For every  $x,y \in A$ with  $|x-y|< 2r$ there exists  $s \in A$ such that
$$   \left| \frac{x+y}{2} -s\right| \leq  r - \sqrt{r^2 - \frac{|x-y|^2}{4}} =: d(x,y,r).$$
\end{enumerate}
\end{lemma}

\begin{remark}\label{ostr}
An elementary computation gives that for $|x-y|< 2r$,
$$  \frac{|x-y|^2}{8r}    \leq    d(x,y,r) \leq  \frac{|x-y|^2}{4r}.$$
\end{remark}

The following easy consequence of a well-known Federer's result is well-known; because of the lack of a reference, we supply the proof.

\begin{lemma}  \label{reach_C11}
Let $U \subset \R^d$ be open and $\Phi: U \to \R^d$ be a $C^{1,1}$-diffeomorphism.
 Assume that $\emptyset \neq C\subset U$ is compact and $\reach C >0$. Then $\reach \Phi(C)>0$. 
\end{lemma}

\begin{proof}
Choose an arbitrary $y_0\in \Phi(C)$, $y_0= \Phi(x_0)$. Choose $\sigma>0$ such that $B(y_0,\sigma) \subset \Phi(U)$
 and $D\Phi^{-1}$ is bounded on  $B(y_0,\sigma)$. Then $\Phi^{-1}$ is Lipschitz on $B(y_0,\sigma)$.
 Now choose  $0< s < \reach C$ such that  $ B(x_0, 2 s) \subset U$ and
  $\Phi( B(x_0, 2 s)) \subset B(y_0,\sigma)$. Then $D \Phi$ is Lipschitz and thus bounded
	 on $B(x_0, 2 s)$; consequently $\Phi$ is Lipschitz on $B(x_0, 2 s)$. By Lemma \ref{ball} 
	$A:= C \cap \overline B(x_0, s)$ has positive reach. Now we easily see that \cite[Theorem~4.19]{Fed59}
	 applied to $A$ and $f:= \Phi|_G$, where $G:= \{x \in \R^d:\ \dist(x,A)< s\}$ yields
	 that $\reach \Phi(A)>0$. Since $\Phi(C) \cap \Phi(B(x_0,s)) = \Phi(A) \cap \Phi(B(x_0, s))$,
	 we obtain  $\reach(\Phi(C),y_0)>0$ (cf. Lemma \ref{local}). Since $C$ is compact and $y_0\in \Phi(C)$
	 was arbitrary, we obtain $\reach \Phi(C)>0$ by \eqref{reach_compact}.
	\end{proof}

\begin{lemma} \label{B_set_PR}
    Any $\B$-set $B\subset\R^2$ has positive reach.
\end{lemma}

\begin{proof}
    By Definition~\ref{B-sets}, $B$ has the form
    $$B:=\{(x,y):\, x\in[p,q],\, \psi(x)\leq y\leq\varphi(x)\},$$
    where $[p,q]=[0,r]$ or $[p,q]=[-r,r]$ for some $r>0$ and $\psi\leq 0\leq\varphi$ are Lipschitz functions on $[p,q]$, $\psi$ is semiconvex on $(p,q)$ and $\varphi$ is semiconcave on $(p,q)$. Then $\psi$ (resp.\ $\varphi$) has a Lipschitz semiconvex (resp.\ semiconcave) extension $\tilde{\psi}$ (resp.\ $\tilde{\varphi}$) defined on $\R$ (see \cite[Proposition~1.7]{Fu85} and Remark~\ref{RemFu}) and we have 
    $$\reach(\epi\tilde{\psi})>0 \text{ and }\reach(\hyp\tilde{\varphi})>0$$
    by \cite[Theorem~2.3]{Fu85}. Take $K>\max\{\max|\varphi|,\max|\psi|\}$ and denote $V:=[p,q]\times [-K,K]$,    $M_+=V\cap\epi\tilde{\psi}$ and $M_-=V\cap\hyp\tilde{\varphi}$; note that $\reach V>0$.
    Now we apply Lemma~\ref{prunik} to show that both $M_+$ and $M_-$ have positive reach. Consider the case $M_+$ (the other being analogous) and let $p\in \partial V\cap \partial(\epi\tilde{\psi})$. Then clearly $e_2\in\Tan(V,p)\cap\INt\Tan(\epi\tilde{\psi},p)$, thus we have $\reach M_+>0$ by Lemma~\ref{prunik}.
    Since $M_+\cup M_-=V$, $M_+\cap\M_-=B$ and $\reach V>0$, we can apply \cite[Theorem~5.16~(5)]{Fed59} and get $\reach B>0$.
\end{proof}

We will also use the following result which essentially goes back to Reshetnyak \cite{Re56} and is an easy consequence of \cite[Theorem~2.6]{Fu85}.

\begin{lemma}  \label{Fu_local}
    Let $u\in S^{d-1}$ ($d\geq 2$), $f:u^\perp\to\R$ be Lipschitz and $a\in\partial(\hyp_uf)$.
    Then $\reach(\hyp_uf,a)>0$ if and only if $f$ is semiconcave on $U:=u^\perp\cap B(\pi_{u^\perp}(a),\delta)$ for some $\delta>0$.
\end{lemma}

\begin{proof}
    We can and will assume that $a=0$ and $u=e_d$, and we denote $W:=e_d^\perp$. Let $K>0$ be such that $f:W\to\R$ is $K$-Lipschitz. Denote $A:=\hyp f$. We will use the usual identifications of $W$ with $\R^{d-1}$ and of $\R^d$ with $W\times\R$.
    
    Assume first that $\reach(A,a)>0$. Using Lemma~\ref{P_PR}~(i) we can find $r>0$ such that $\reach(A,x)>r$ whenever $|x|<r$. Set $\delta:=r/\sqrt{1+K^2}$ and $U:=W\cap B(0,\delta)$. Consider any $w\in U$. Then, denoting $x:=(w,f(w))\in\partial A$, we have $|x|<r$ since $f$ is $K$-Lipschitz. By \cite[Proposition~3.1~(vi)]{RZ17} there exists $n\in\Nor(A,x)\cap S^{d-1}$ and we have $\overline{B}(x+rn,r)\cap A=\{x\}$ by \cite[Lemma~4.3]{RZ19}. Thus the assumptions of \cite[Theorem~2.6]{Fu85} are satisfied and we obtain that $f$ is semiconcave on $U$ (recall Remark~\ref{RemFu}).

     Assume now that $f$ is semiconcave on $W\cap B(0,\delta)$ for some $\delta>0$. Then, the restriction $f|_{W\cap B(0,\delta)}$ has a Lipschitz semiconcave extension $g:W\to\R$ (see \cite[Proposition~1.7]{Fu85}) and we have $\reach\hyp g>0$ by \cite[Corollary~2.8]{Fu85} (again, recall Remark~\ref{RemFu}). Since
    $\hyp g\cap B(0,\delta)=\hyp f\cap B(0,\delta)$, 
    we get $\reach(\hyp f,0)>0$ by Lemma~\ref{local}. This proves the second implication.
\end{proof}

\section{Lipschitzness of $\psi^A_k$ on some subsets of $T_k(A)$}

Our approach (for proving the  Lipschitzness of $\psi^A_k$) is 
based on the notion of fullness of simplices and on the tangential regularity of sets with positive reach, and it
is more elementary and also works in more cases than that of \cite{Ly24}.
The main idea is contained in the proof of Lemma~\ref{H}. 
First we present some definitions and several simple lemmas.

\begin{definition}\label{gap}
	Let $1\leq k\leq d-1$ and $U,V\in G(d,k)$. We set
	 $$  \rho_k(U,V) = \max\left( \sup_{u \in U\cap S^{d-1}} \dist(u,V), \sup_{v \in V \cap S^{d-1}} \dist(v,U)\right).$$
\end{definition}
However, it is well-known that in fact,
\begin{equation}  \label{gap2}
    \rho_k(U,V) = \sup_{u \in U\cap S^{d-1}} \dist(u,V),
\end{equation}
for which we know a direct reference only in the complex case (\cite[Lemma~3.2]{Morris}); however the real case is e.g.\ an obvious consequence of \cite[Lemma~2]{HRW13}.
Recall also (see \cite[Theorems~2.251, 2.250]{ThI}) that
\begin{equation}\label{vlga}
	\rho_k(U,V)= \|\pi_U-\pi_V\|\, \text{ and }\, \rho_{d-k}(U^{\perp},V^{\perp}) = \rho_k(U,V).
\end{equation}
It follows that $\rho_k$ is a metric on $G(d,k)$. It is also well-known that $G(d,k)$ with $\rho_k$ is  compact (see e.g.\ \cite[Proposition~2.253 and (2.91)]{ThI}).

If $A\subset \R^d$ is a set of positive reach, we will often classify its points $x$ according the magnitude of tangent cone $\Tan(A,x)$. For the short, we introduce the following notation.

Given $A\subset\R^d$ with positive reach, we set
$$\wtilde \Tan(A,x):= \spa \Tan(A,x),\quad x\in A.$$
Further, for $k=0,1,\dots,d$, we will use the notation
\begin{align*}
    T_k(A)&:= \{x\in A:\ \dim (\Tan(A,x))=k\},\\
    T_k^+(A)&:= \{x\in T_k(A):\ \wtilde \Tan(A,x)= \Tan(A,x) \},\text{ and}\\
    T_k^-(A)&:= T_k(A) \setminus T_k^+(A).
\end{align*}
Note that if $\dim A=k$ then, using Lemma~\ref{P_PR}~(ii),(iv) we have
\begin{equation}   \label{tan_k}
    A=T_0(A)\cup T_1(A)\cup\dots\cup T_k(A).
\end{equation}
We also define the mapping  $\psi_k^A: T_k(A) \to G(d,k)$ by
$$   \psi_k^A(x):= \wtilde \Tan(A,x),\ \ x \in T_k(A).$$
On   $G(d,k)$  we always consider the metric $\rho_k$ (see Definition \ref{gap}).
Following  \cite{Ly24}, we will look for subsets of $T_k(A)$ on which the mapping $\psi_k^A$ is Lipschitz.

We will need the following two easy lemmas.

\begin{lemma}\label{tkdif}
Let $U\subset \R^d$ be open, $\Phi: U \to \R^d$ a $C^1$ diffeomerphism, $a\in A \subset U$ and $0\leq k \leq d$. Then
\begin{enumerate}
	\item[{\rm (i)}]
	$\Tan(\Phi(A), \Phi(a))=D\Phi(a)(\Tan(A,a)),$
	\item[{\rm (ii)}]
	$\wtilde\Tan(\Phi(A), \Phi(a))=D\Phi(a)(\wtilde\Tan(A,a)),$
	\item[{\rm (iii)}]
	$T_k(\Phi(A))=  \Phi(T_k(A))$ if $\reach A>0$.
\end{enumerate}
\end{lemma}

\begin{proof}
    Statement (i) follows from \cite[\S 3.1.21]{Fed69}. Assertions (ii) and (iii) can be easily obtained from (i).
\end{proof}

\begin{lemma}   \label{T_k}
    Let $A\subset\R^d$ have positive reach and $1\leq k\leq d$. Then
    \begin{enumerate}
        \item[{\rm (i)}] $T_k(A)\cap\overline{\bigcup_{i=0}^{k-1}T_i(A)}=\emptyset$,
        \item[{\rm (ii)}] the mapping $\psi_k^A$ is continuous on $T_k(A)$,
        \item[{\rm (iii)}] $T_1(A)$ is closed,
        \item[{\rm (iv)}] $T_0(A)$ is the set of all isolated points of $A$.
    \end{enumerate}
\end{lemma}

\begin{proof}
    Let $x_0\in T_k(A)$. We can choose $k$ linearly independent vectors $v_1,\dots,v_k\in\Tan(A,x_0)$; these vectors form a basis of $\widetilde{\Tan}(A,x_0)$. Now consider an arbitrary $\ep>0$. By \cite[Lemma~2.252]{ThI} there exists $\delta>0$ such that
    \begin{align}   \label{thib}
        \text{if }w_i\in B(v_i,\delta),\, &i=1,\dots,k,\text{ then } \spa(w_1,\dots,w_k)\in G(d,k)\\
        &\text{and }\rho_k\left(\spa(w_1,\dots,w_k),\widetilde{\Tan}(A,x_0)\right)<\ep. \nonumber
    \end{align}
    Since the multivalued mapping $x\mapsto\Tan(A,x)$ is lower semicontinuous (see Lemma~\ref{P_PR}~(vii)), we can choose $\omega>0$ such that for each $x\in B(x_0,\omega)$ there exist vectors
    $w_i\in B(v_i,\delta)\cap\Tan(A,x)$, $i=1,\dots,k$. By \eqref{thib} we obtain $\spa(w_1,\dots,w_k)\in G(d,k)$ and, hence, $x\not\in\overline{\bigcup_{i=0}^{k-1}T_i(A)}$. This proves (i).

    Further, if $x\in B(x_0,\omega)\cap T_k(A)$ then $\spa(w_1,\dots,w_k)=\widetilde{\Tan}(A,x)$ and so \eqref{thib} gives 
    $\rho_k\left(\psi_k^A(x),\psi_k^A(x_0)\right)<\ep$. Thus we have proved the continuity of $\psi_k^A$ at $x_0$ and (ii) follows.

    Assertion (iv) follows easily from the definition of tangent vectors. If $a\in\overline{T_1(A)}$ then obviously $a\in A$ and $a\not\in T_0(A)$ by (iv). Further, by (i), $a\not\in T_k(A)$ for each $1<k\leq d$, hence $a\in T_1(A)$. This proves (iii). 
\end{proof}

We easily obtain that $\psi_1^A$ is globally Lipschitz:

\begin{proposition}  \label{k=1}
    Assume that $A\subset\R^d$ and $0<r<\reach A$. Then $\psi_1^A$ is $\frac 1r$-Lipschitz on $T_1(A)$.
\end{proposition}

\begin{proof}  
    Let $a,b\in T_1(A)$, $a\neq b$. Using Proposition~\ref{fedtan} we find that there exists $u\in\Tan(A,a)$ such that $|b-a-u|\leq|b-a|^2/2r$, hence 
$$\left|\frac{b-a}{|b-a|}-\frac{u}{|b-a|}\right|=\left|\frac{a-b}{|b-a|}-\frac{-u}{|b-a|}\right|\leq\frac{|b-a|}{2r}.$$ 
Consequently, using \eqref{gap2} we have $\rho_1(\spa\{b-a\},\widetilde{\Tan}(A,a))\leq \frac{|b-a|}{2r}.$ Analogously we obtain $\rho_1(\spa\{a-b\},\widetilde{\Tan}(A,b))\leq \frac{|b-a|}{2r}$, and since $\spa\{a-b\}=\spa\{b-a\}$, the assertion follows.
\end{proof}  

However, in the case $k\geq 2$, $\psi_k^A$ is only localy Lipschitz and the proof is much more involved; we will need several auxiliary notions and lemmas.
	
Let  $2\leq k\leq d$ be integers. If $a_0,a_1,\dots,a_k \in \R^d$, we define
$$\sigma(a_0,\dots, a_k):= \conv\{a_0,\dots,a_k\}$$ 
and its  $k$-dimensional volume will
	 be denoted by  $|\sigma(a_0,\dots, a_k)  |$. If $a_0,a_1,\dots,a_k$ are affinely independent,
	 then $\sigma(a_0,\dots, a_k)$ is the simplex with vertices  $a_0,a_1,\dots,a_k$.

Following  \cite[p.~125]{Whi57} we define the  {\it fullness} of  $\sigma= \sigma(a_0,\dots, a_k)$  with $\diam\sigma>0$ as
	\begin{equation}\label{ful}
	  \Theta(\sigma):=  \frac{|\sigma|}{(\diam\sigma)^k}.
		\end{equation}
		The notion of fullness (which is ``almost equivalent'' to the more frequent notion of ``thickness'')
		 will be useful for us; we will use its  following  properties from \cite{Whi57}.

\begin{lemma}\label{otlus}
	Let  $2\leq k\leq d$ and  $\sigma= \sigma(a_0,\dots, a_k) \subset \R^d$ be given. Then
	 the following assertions hold.
\begin{enumerate}
\item[{\rm (i)}] $\diam\sigma= \max\{|a_i-a_j|:\ 0\leq i<j \leq k\}$, $\Theta(\sigma) \leq \frac{1}{k!}$. 
\item[{\rm (ii)}] $\min\{|a_i-a_j|:\ 0\leq i<j \leq k \} \geq k!\,  \Theta(\sigma)\,  \diam\sigma.$
\item[{\rm (iii)}] If $\sigma$ is similar to a simplex $\sigma^*$, then $\Theta(\sigma)= \Theta(\sigma^*)$.
\item[{\rm (iv)}] If  $\sigma^p= \sigma(a^p_0,\dots, a^p_k) \subset \R^d$, $p\in \N$, 
and $a_j^p  \to a_j,\ j=0,\dots, k$, then $\Theta(\sigma^p) \to \Theta(\sigma)$.
\item[{\rm (v)}] If $\Theta(\sigma)>0$ and $\lambda_1,\dots,\lambda_k$ are real numbers, then
$$ |\lambda_i| \leq  \frac{|\sum_{j=1}^{k} \lambda_j (a_j-a_0)|}{k! \Theta(\sigma) |a_i-a_0|},\quad	 i=1,\dots, k.$$
\end{enumerate}
\end{lemma}
	
\begin{proof}
	The first part of (i) is easy and well-known, for the second one see \cite[p.~125, (3)]{Whi57}.
	For (ii) see \cite[p.~126, (5)]{Whi57} and (iii) is obvious. The proof of (iv) is easy, cf.\
	  \cite[p.~126, Lemma~14c]{Whi57} (and the note before it). To prove (v), set $u_i:= (a_i-a_0)/|a_i-a_0|$
		  and  $\lambda_i^*:= \lambda_i |a_i-a_0|$. Then  \cite[p.~127, (3)]{Whi57} gives
			 $ |\lambda_i^*| \leq  |\sum_{j=1}^{k} \lambda_j^* u_j|/k! \Theta(\sigma)$ and
			the assertion of (v) follows.
\end{proof}

	Moreover, we need the following lemma.
	\begin{lemma}\label{vymena}
	Let $2\leq k \leq d$, $\theta>0$, $v_1,\dots,v_k, w \in S^{d-1}$, $\sigma= \sigma(0,v_1,\dots,v_k)$ and
	 $\Theta(\sigma) \geq \theta$. Then there exists  $1\leq i \leq k$ such that
	 $$\Theta(\sigma(0,v_1,\dots,v_{i-1},w,v_{i+1},\dots,v_k)) \geq  \theta/(k 2^{k+1}).$$
	\end{lemma}
	\begin{proof}  
	Denote  $L:= \spa\{v_1,\dots,v_k\}$, $\tilde w:= \pi_L(w)$ and $z:= w- \tilde w \in L^{\perp}$.
	Write  $\tilde w = \sum_{j=1}^k a_j v_j$ and choose $1\leq i \leq k$ with  $|a_i|= \max(|a_1|,\dots, |a_k|)$.
	Denote
	\begin{align*} \sigma^-&:=  \sigma(0,v_1,\dots,v_{i-1}, v _{i+1},\dots, v_k), \\ 
	\sigma_w&:= \sigma(0,v_1,\dots,v_{i-1},w, v _{i+1},\dots, v_k), \\
	\sigma_{\tilde w}&:= \sigma(0,v_1,\dots,v_{i-1},\tilde w, v _{i+1},\dots, v_k).
	\end{align*}
	Then  $\sigma_{\tilde w}= \pi_L(\sigma_w)$ and consequently  $|\sigma_w| \geq |\sigma_{\tilde w}|$.
	Since  $w=z + \tilde w$, we have either  $|z| \geq 1/2$ or $|\tilde w| \geq 1/2$.
	
	First consider the case  $|z| \geq 1/2$. Then, since clearly  
	$|z| \leq \dist(w, \spa \sigma^-)$ and   $|\sigma| \leq |\sigma^-|$, we obtain 
	$$|\sigma_w|=(1/k)|\sigma^-|\,\dist(w, \spa \sigma^-) \geq (1/k)|\sigma^-|\cdot |z| \geq (1/2k) |\sigma|.$$ 
	In the case  $|\tilde w| \geq 1/2$ we obtain  
	$$  k |a_i| \geq  \sum_{j=1}^k |a_j|  \geq  |\tilde w| \geq  \frac{1}{2}$$
	 and consequently  $|a_i| \geq  1/(2k)$. 
	 Now choose  $u \in L \cap S^{d-1}$ for which  $u \perp v_j$,\ $j \neq i$. Then clearly
	$$  |\sigma| = \frac{1}{k}  |\sigma^-| \cdot | \langle u, v_i \rangle|,\ \ 
	|\sigma_{\tilde w}| = \frac{1}{k}  |\sigma^-| \cdot | \langle u, \tilde w \rangle| $$
	 and so, using also the equalities
	$$ | \langle u, \tilde w \rangle|  = \left| \langle u,\sum_{j=1}^k a_j v_j \rangle \right | =
	 | \langle u, v_i \rangle|\cdot |a_i|,$$
	 we obtain  $|\sigma_w|\geq |\sigma_{\tilde w}| = |a_i|\cdot |\sigma| \geq (1/2k)|\sigma|$.
	So, since $\diam \sigma_{ w} \leq 2$ and $\diam \sigma \geq 1$ by Lemma~\ref{otlus}~(i),
	 we obtain in both cases
\begin{align*}
  \Theta(\sigma_{w})&= \frac{|\sigma_{ w}|}{(\diam \sigma_{w})^k}    
	\geq\frac{(1/2k)|\sigma|}{(\diam \sigma_{w})^k}\\
    &\geq 
	 \frac{|\sigma|}{2k\cdot 2^k (\diam \sigma)^k} = \frac{\Theta(\sigma)}{k2^{k+1}}\geq 
	\frac{\theta}{k2^{k+1}}.
\end{align*}
	 \end{proof}

\begin{definition}\label{related}
Let   $2\leq k\leq d$, $A \subset \R^d$ and $\theta >0$. We say that the points $a,b \in \R^d$, $a\neq b$, are 
\emph{$(A,k,\theta)$-related} if there exist  points  $z_1,\dots, z_{k-1} \in A$ such that
 $\Theta(\sigma(a,b,z_1,\dots, z_{k-1})) \geq \theta$.
\end{definition}

	\begin{lemma}\label{H}
	Let $2\leq k < d$, 
	 $A \subset \R^d$, $0< r < \reach A$, $\theta>0$, and let $a,b \in T_k(A)$ be
	 $(A,k,\theta)$-related. Then
	$$ \rho_k(\psi_k^A(a), \psi_k^A(b)) \leq L_{k,\theta,r} |b-a|,$$
	 where   $L_{k,\theta,r}:= \frac{k}{ (k! \theta)^2 r}$  .
	\end{lemma}
	\begin{proof}
	Choose    points  $z_1,\dots, z_{k-1} \subset A$ such that
 $\Theta(\sigma(a,b,z_1,\dots, z_{k-1})) \geq \theta$. Denote $z_0:=a$, $z_k:=b$ and
 $V:=  \spa( z_1-z_0,\dots, z_k-z_0)$.  By  Proposition  \ref{fedtan}
 there exist vectors  $v_1,\dots,v_k \in \Tan(A,z_0)$ such that
 \begin{equation}\label{aplfed}
|(z_j-z_0)-v_j| \leq  \frac{|z_j-z_0|^2}{2r},\quad j=1,\dots, k.
\end{equation}
 Now consider an arbitrary $u \in V \cap S^{d-1}$, write
	\begin{equation}\label{utt}
	u =  \sum_{j=1}^{k}    t_j (z_j -z_0)\ \ \ \text{and set}\ \ \  v :=  \sum_{j=1}^{k}    t_j v_j.
	\end{equation}
	By   Lemma \ref{otlus}  (v)  and  (ii)  we obtain that
$$   |t_j| \leq \frac{1}{k! \theta |z_j-z_0|}\ \ \ \text{and}\ \ \ |b-a|  \geq  k! \theta   |z_j-z_0|, \quad j=1,\dots,k.$$	  
Using these inequalities and	\eqref{aplfed} we obtain
\begin{multline*}
|u-v|   \leq  \sum_{j=1}^{k}   | t_j|   |(z_j -z_0)-v_j| \leq \sum_{j=1}^{k}  |t_j| \frac{|z_j-z_0|^2}{2r} \\
 \leq    \sum_{j=1}^{k}   \frac{  |z_j-z_0|}{2 k! \theta r} \leq \frac{k}{ 2(k! \theta)^2 r} |b-a|.
\end{multline*}	
Therefore, using that $u\in V\cap S^{d-1}$ was arbitrary,  $v \in \wtilde \Tan(A,a)$ and \eqref{gap2}, we obtain
	$$\rho_k(\wtilde \Tan(A,a), V) \leq  \frac{L_{k,\theta,r}}{2} |b-a|.$$
By the same way (setting $z_0:=b$ and $z_k:=a$) we obtain
	$$\rho_k(\wtilde \Tan(A,b), V) \leq  \frac{L_{k,\theta,r}}{2} |b-a|$$ 
	 and consequently  $\rho_k(\wtilde \Tan(A,a), \wtilde \Tan(A,b)) \leq L_{k,\theta,r} |b-a|$.
	\end{proof}

	\begin{lemma}\label{omR}
	Let $d \geq 3$, $2\leq k \leq d-1$ and $\theta$>0.  
	If  $A\subset \R^d$ has positive reach, $x \in A$,  
	 $v_1,\dots, v_k \in \Tan(A,x)\cap S^{d-1}$ and   $\Theta(\sigma(0,v_1,\dots,v_k)) \geq \theta$,
	 then there exists  $\delta>0$ such that if $a,b \in  B(x,\delta) \cap A$ and  $a \neq b$,
	 then $a,b$ are $(A,k,2^{-(k+2)}k^{-1}\theta)$-related.
	\end{lemma}
	\begin{proof}
	Suppose, to the contrary, that such  $\delta>0$ does not exist. Then we can choose, for each $n \in \N$,
	 points  $a_n, b_n \in B(x, 1/n) \cap A$ such that  $a_n \neq b_n$ and
	\begin{equation}\label{anbn}
	\text{the points}\ \ a_n,b_n\ \ \text{are not}\ \  (A,k, 2^{-(k+2)}k^{-1}\theta)\text{-related}.
	\end{equation}
	Now we can choose subsequences $(a_{n_p})$, $(b_{n_p})$ and $w \in S^{d-1}$ such  that
	$ (b_{n_p}-a_{n_p})/|b_{n_p}-a_{n_p}| \to w \in S^{d-1}$. 
	By Lemma \ref{vymena} there exists  $1\leq j \leq k$ such that
	 $\Theta(\sigma(0,v_1,\dots,v_{j-1}, w, v_{j+1},\dots,v_k)) \geq   2^{-(k+1)}k^{-1}\theta$. Without any loss of generality
	 we can suppose that $j=k$; so we have
	\begin{equation}\label{povym}
	\Theta(\sigma(0,w,v_1,\dots,v_{k-1})) \geq  2^{-(k+1)}k^{-1}\theta.
	\end{equation}
Set  $\tau_p:= |b_{n_p}-a_{n_p}|$, $p\in \N$. Then $\tau_p>0$ and $\tau_p \to 0$.  Since $v_1,\dots v_{k-1}$ are Clarke tangent vectors of $A$ at $x$ by Lemma~\ref{P_PR}~(vi) and  $a_{n_p}\in A \to x$, by  \eqref{Clarke}
	 there exist points $c^j_p \in A,\ 1\leq j \leq k-1,\  p\in \N$ such that  $(c^j_p - a_{n_p})(\tau_p)^{-1} \to v_j,\   1\leq j \leq k-1.$
	 Then 
	$ \sigma(a_{n_p}, b_{n_p},c^1_p,\dots, c^{k-1}_p )$ is similar to
	$\sigma(0,(b_{n_p}-a_{n_p})/\tau_p, (c^1_p-a_{n_p})/\tau_p, \dots ,(c^{k-1}_p-a_{n_p})/\tau_p)  $
	 whose vertices converge to the vertices of $\sigma(0,w,v_1,\dots,v_{k-1})$ with $p\to\infty$. Using \eqref{povym} and
	   Lemma~\ref{otlus}~(iii), (iv)   we obtain that there exists $p \in \N$ 
	 such that  $ \Theta(\sigma(a_{n_p}, b_{n_p},c^1_p,\dots, c^{k-1}_p )) \geq 2^{-(k+2)}k^{-1}\theta$  and  thus the points 
	 $a_{n_p}, b_{n_p}$ are  $(A,k, 2^{-(k+2)}k^{-1}\theta)$-related which contradicts \eqref{anbn}.	
	\end{proof}

	 As an easy consequence of Lemma \ref{omR}  and  Lemma \ref{H}  
	we obtain the following result.

	\begin{proposition}\label{omL}
	Let $d \geq 3$, $2\leq k \leq d-1$ and $\theta$>0. 
	Suppose that $A \subset \R^d$, $0<r< \reach A$, $a \in A$,  and there exist
	 $v_1,\dots, v_k \in \Tan(A,a)$ such that   $\Theta(0,v_1,\dots,v_k) \geq \theta$.
	 Then there exists $\delta>0$ such that the mapping $\psi_k^A$ is   $\tilde L_{k,\theta,r}$ -Lipschitz
	 on  $T_k(A) \cap B(a,\delta)$,
	where  $\tilde L_{k,\theta,r}:= 2^{2(k+2)} k^3 (k!)^{-2} \theta^{-2} r^{-1}$.  
	 \end{proposition}
	\begin{proof}
	Let $\delta>0$ be as in Lemma \ref{omR}. Then it is sufficient 
	to apply Lemma \ref{H}  (with  ``$\theta:= 2^{-(k+2)}k^{-1}\theta$'').
	\end{proof}
	
Since  for each  $a \in T_k(A)$ ($k\geq 2$) we can choose linearly independent vectors $v_1,\dots, v_k \in \Tan(A,a)$ and then   $\Theta(\sigma(0,v_1,\dots,v_k)) >0$, the above ``quantitative'' result together with Proposition~\ref{k=1} (for $k=1$) immediately imply the following interesting ``qualitative'' result.
	
\begin{theorem}\label{Tkloklip}
    Let $d \geq 2$ and $1\leq k \leq d-1$.
    Suppose that $A \subset \R^d$ has positive reach and $a\in T_k(A)$.  Then there exists $\delta>0$ such that the mapping $\psi_k^A$ is Lipschitz on  $T_k(A) \cap B(a,\delta)$.
\end{theorem}
    
\begin{remark}
    As an easy consequence we obtain using \cite[Proposition~A48]{Lee} that $\psi_k$ is Lipschitz on any compact subset of $T_k(A)$.
\end{remark} 

However, the following lemma (which is important in the proof of our main result) needs a new additional idea.

\begin{lemma}\label{lipna2}
	Let  $A\subset \R^d$, $0<r< \reach A$,
	$\diam A < r$ and 
	 $A \subset T_1^-(A)  \cup   T_2(A)$. 
	 Then  $\psi_2:=\psi_2^A$  is $(2^{12}\pi/r)$-Lipschitz on $T_2(A)$.
\end{lemma}

\begin{proof}
	Consider two different points  $y,z \in  T_2(A)$. By Lemma \ref{krivky}
					 there exists
	 a curve $\vf: [0,l] \to  A$ parametrized by arclength such that  $\vf(0)=y$,
	 $\vf(l)=z$,  $\vf$ is $C^{1,1}$ smooth on $[0,l]$ and   $l\leq \tfrac{\pi}{2}|z-y|$. Note that  $\vf$ is $1$-Lipschitz. 
	 Now choose an arbitrary  $t \in (0,l)$ and denote $x:= \vf(t)$. 
	Since $\varphi$ is parametrized by arclength and is $C^1$ smooth, we have  $\{\vf'(t),- \vf'(t)\} \subset  \Tan(A,x) \cap  S^{d-1}$, hence 
	  $\Tan(A,x)$ contains a line. Therefore $x\notin T_1^-(A)$, so $x\in T_2(A)$ and 
	$\Tan(A,x)$ contains a closed halfplane. Consequently  $\Tan(A,x)$ contains two orthogonal unit
	 vectors  $v_1$, $v_2$.  An elementary planar calculation now gives
			 $\Theta(\sigma(0,v_1,v_2))= 1/4$. So, applying Proposition \ref{omL} (with $k=2$ and $\theta=1/4$)
			 there exists  $\delta_x>0$ such that  $\psi_2$ is $L$-Lipschitz on  $T_2(A) \cap B(x,\delta_x)$,
			 where  $L:= 2^{13}/r$.  Since  $x\in \vf((0,l))$  was arbitrary, we easily
			 see that  $\psi_2 \circ \vf$
	 is $L$-locally Lipschitz on $(0,l)$.
	It follows (use, e.g. \cite{Fed69}, 2.2.7) that $\psi_2 \circ \vf$ is $L$-Lipschitz on
	 $(0,l)$. Since $\psi_2$ is continuous on $T_2(A)$  by  Lemma~\ref{T_k}~(ii), we obtain that $\psi_2 \circ \vf$ 
		 is continuous on $[0,l]$ and thus  $L$-Lipschitz on
	 $[0,l]$. Consequently
	$$   \rho_2(\psi_2(y),\psi_2(z))= \rho_2(\psi_2(\vf(0)),\psi_2(\vf(l))) \leq L\cdot l
	 \leq  \tfrac{2^{12}\,\pi}{r}|y-z|,$$
which completes the proof.
\end{proof}

\section{Local structure of $T_k(A)$}

Using Theorem~\ref{Tkloklip} and the ``$C^{1,1}$ Whitney theorem'' (Fact~\ref{Wh}), we obtain the following result.

\begin{theorem}\label{tdmnapl2}
	Let $d\geq 2$, $1\leq k \leq d-1$,  $A \subset \R^d$ have positive reach and $a\in T_k(A)$. 
	Then there exists $\delta >0$
	 such that  $T_k(A) \cap B(a, \delta)$ is a subset of a   $k$-dimensional $C^{1,1}$ surface
	 $\Gamma$.
	% Moreover, for each $c \in T_k(A) \cap B(a, \delta)$, we have that $\wtilde \Tan(A,c)$ is the tangent plane
	 %to $\Gamma$ at $c$.
\end{theorem}

\begin{proof}
Without any loss of generality we will suppose that $a=0$.
Choose $0< r  < \reach A$.
By Theorem  \ref{Tkloklip} there exists $\omega>0$ and $L>0$ such that  $\psi_k=\psi_k^A$
		 is $L$-Lipschitz on  $M:= T_k(A) \cap B(0,\omega)$. Denote $T:= \psi_k(0)$ and  
	 define  functions $f:M\to\R^d$ and $\varphi:M\to\cL(\R^d,\R^d)$ as follows:
	 $$ f(x):= \pi_T(x)\ \ \ \text{and}\ \ \ \vf(x):= \pi_T + \pi_{\psi_k(x)^{\perp}},\quad x\in M.$$
    We will apply Fact~\ref{Wh} (Whitney's $C^{1,1}$ extension theorem) to $f$ and $\vf$. For this purpose, we have to verify conditions \eqref{Whitney_1} and \eqref{Whitney_2}.
   Consider arbitrary points $x,y\in M$. We get
    \begin{align*}
        |f(y)-&f(x)-\varphi(x)(y-x)|=|(\pi_{T}(y)-\pi_{T}(x))-\varphi(x)(y-x)|\\
        &=|\pi_{\psi_k(x)^\perp}(y-x)|
        \leq \frac{|y-x|^2}{2r},
    \end{align*}
    where the last estimate follows from Proposition~\ref{fedtan}, since 
    $$|\pi_{\psi_k(x)^\perp}(y-x)|=\dist(y-x,\psi_k(x))\leq \dist(y-x,\Tan(A,x)).$$ 
    Also,
    $$ \|\varphi(y)-\varphi(x)\|= \|\pi_{\psi_k(y)^\perp}-\pi_{\psi_k(x)^\perp}\|
        = \rho_k(\psi_k(x),\psi_k(y))\leq L|y-x|,$$
    where \eqref{vlga} and the $L$-Lipschitzness of $\psi_k$ on $M$ were used, so \eqref{Whitney_1} and \eqref{Whitney_2} follow with $c:=\max\{L,\frac{1}{2r}\}$.
    Now, using Fact~\ref{Wh}, there exists a $C^{1,1}$ mapping $F:\R^d\to\R^d$ such that $F(x)=f(x)$
	and $DF(x)=\varphi(x)$ whenever $x\in M$.
	Note that $F(0)=f(0)=0$.
	  Since $DF(0)=\varphi(0)$ is the identity map, we can apply Fact~\ref{locdif} to $F$ and find  
		$\omega>\ep>0$ such that the restriction  $\Phi|_{B(0,\ep)}$ is a $C^{1,1}$-diffeomorphism. Choose $\ep >\delta>0$
		 such that $B(0,\delta) \subset \Phi(B(0,\ep))$ and set  $G(x):= \pi_{T^{\perp}}(\Phi^{-1} (x))$ for
		 $x \in T \cap B(0,\delta)$. Then $G: T\cap B(0,\delta) \to T^{\perp}$ is a $C^{1,1}$ mapping  which
		  has a $C^{1,1}$ extension  $\tilde G: T \to T^{\perp}$ by Fact~\ref{Wh2}. Set    
			 $\Gamma:= \{x+ \tilde G(x):\ x \in T\}$. Then $\Gamma $ is 
			 a   $k$-dimensional $C^{1,1}$ g-surface    and  $T_k(A) \cap B(a, \delta) \subset \Gamma$.
    	Indeed, consider an arbitrary $z \in T_k(A) \cap B(a, \delta)\subset M$  and denote $x:= \Phi(z) = f(z) = \pi_T(z)$.
			 Then $G(x) =  \pi_{T^{\perp}}(\Phi^{-1} (x)) = \pi_{T^{\perp}}(z)$ and so
			$z = \pi_T(z) +  \pi_{T^{\perp}}(z) =   x+ G(x) \in \Gamma$.
\end{proof}

As a direct consequence we obtain the following corollary which is new up to our knowledge. 
    
\begin{corollary}  \label{count}
    Let $A\subset\R^d$ be a set with positive reach and $1\leq k< d$. Then 
    \begin{enumerate}
        \item[{\rm (i)}] $T_k(A)$ can be covered by countably many $k$-dimensional $C^{1,1}$ surfaces,
        \item[{\rm (ii)}] if $\dim A= k$ then $A$ can be covered by countably many $k$-dimensional $C^{1,1}$ surfaces.
    \end{enumerate}
\end{corollary}

\begin{proof}
    By Theorem~\ref{tdmnapl2} and Definition~\ref{plochy}, we can assign to each $a\in T_k(A)$ a $\delta_a>0$ and a $k$-dimensional $C^{1,1}$ g-surface $\Gamma_a$ such that $T_k(A)\cap B(a,\delta_a)\subset\Gamma_a$. Since $\R^d$ is a separable metric space, there exists a countable set $S\subset T_k(A)$ such that $\bigcup_{a\in T_k(A)}B(a,\delta_a)=\bigcup_{a\in S}B(a,\delta_a)$. Then 
    $$T_k(A)\subset\bigcup_{a\in S}\left( T_k(A)\cap B(a,\delta_a)\right)\subset\bigcup_{a\in S}\Gamma_a.$$
    This proves (i). Further, we obtain (ii) from (i) using \eqref{tan_k} since each $i$-dimensional $C^{1,1}$ g-surface with $1\leq i\leq k$ is clearly contained in a $k$-dimensional $C^{1,1}$ g-surface, and $T_0(A)$ is countable by Lemma~\ref{T_k}~(iv).
\end{proof}

\begin{remark}
    By \cite[Theorem~7.5]{RZ17}, if $\dim A=k$ then $A$ can be locally covered by \emph{finitely} many ``DC surfaces of dimension $k$'' (which are even semiconcave if $k=d-1$, see \cite[Theorem~5.9]{RZ17}). These results, however, do not imply Corollary~\ref{count}~(ii) where $C^{1,1}$ surfaces are used.
\end{remark}

In the case $k=1$ we obtain (using a different method) a stronger version of Theorem~\ref{tdmnapl2} with $\Gamma$ being locally contained in $A$. This will be used in the proof of our main result (Theorem~\ref{T_main}).

\begin{theorem}\label{T1+}
	Let $d\geq 2$,  $A \subset \R^d$ have positive reach and $a\in T_1^+(A)$. Then there exists 
	a  $1$-dimensional $C^{1,1}$ surface
	 $\Gamma$ and
	$\delta >0$
	 such that 
	\begin{equation}\label{intj}
	 T_1(A)\cap B(a, \delta) \subset \Gamma \cap B(a, \delta) \subset  A.
	\end{equation}
\end{theorem}

\begin{proof}
    Without any loss of generality we will suppose that $a=0$  and  $e_1 \in \Tan(A,0)$.
    Denote  $W:= \spa{e_1}$ and  $\pi:= \pi_W$.
		
		Fix  $0<r< \reach A$. Now choose  $0< \omega < r/2$  and put  $D:=  A \cap  \overline B(0,\omega)$.
		 By Lemma~\ref{ball}  we have  $r < \reach D$.
		
   Using  Corollary~\ref{uhlovy}  (with $A:=D$, $a:=0$, $b:=z$), we obtain that 
    \begin{equation}\label{vnule}
    \angle(z, e_1) < \pi/12\ \ \text{ or}\ \  \angle(z,- e_1) < \pi/12,\ \ \text{ whenever}\ \ 
		z\in D\setminus \{0\}.
    \end{equation}
     Now consider an arbitrary $z \in T_1(D) \setminus \{0\}$. By  Corollary  \ref{uhlovy}
		(used with $A:=D$, $a:=z$, $b:=0$)
		we obtain that there 
    exists  $v \in \Tan(D,z)$ with  $\angle(-z,v) < \pi/12$ and therefore \eqref{vnule}  gives that either
    $\angle(v,e_1) < \pi/6$ or  $\angle(v,-e_1) < \pi/6$. 
    Now, if $\tilde z \in D \setminus \{z\}$,  Corollary  \ref{uhlovy} (used with $A:=D$, $a:= z$, $b:= \tilde z$) gives that either 
     $\angle(\tilde z - z, v) < \pi/6$ or  $\angle(\tilde z - z,- v) < \pi/6$ and
    therefore   either 
    $\angle(\tilde z - z, e_1) < \pi/3$ or  $\angle(\tilde z - z,- e_1) < \pi/3$. Consequently we obtain 
    \begin{equation}\label{prnad}
     \pi^{-1}(\{\pi(z)\}) \cap D = \{z\}\ \ \text{for each}\ \ z \in T_1(D).
    \end{equation} 
    Since $\Tan(D,0)= W$, we can choose points  $z_1, z_2  \in D$ such that  $x_1:= \pi (z_1) = t_1 e_1$ with  $t_1<0$ and  $x_2:= \pi (z_2) = t_2 e_1$ 
    with  $t_2>0$. Since  $|z_1-z_2|< 2 r$,  Lemma~\ref{krivky} implies that there exists a $C^{1,1}$ curve $\gamma:[0,l]\to D$ parametrized by arclength with $\gamma(0)=z_1$ and $\gamma(l)=z_2$. 
    Using  \eqref{prnad} and connectivity of $\gamma$ we easily obtain that
 	$0=\gamma(t)$ for some $t\in (0,l)$.
    Using that $|\gamma'(t)|=1$, $\Tan(D,0)= W$ and $\gamma([0,l])\subset D$ we infer  $\gamma'(t)=\pm e_1$. 
		Further we will use the standard identification $W=\R$.
		We apply now Fact~\ref{locdif} to the real $C^{1,1}$ function $G:=\pi\circ\gamma|_{(0,l)}: (0,l)\to W=\R$ (with $G'(t)=\pm 1$) and obtain that for some open interval $U\subset \R$ containing $t$, $G|_U:U\to\R$ is a $C^{1,1}$-diffeomorphism. 
		Obviously,
		$$ \pi\left(\gamma\left ( \left (G|_U\right)^{-1}(w)\right)\right)=
		G\left( \left (G|_U\right)^{-1}(w)\right)=w,\ \ \
		 w\in G(U) \subset W,$$
		and so
		$$ \varphi(w):=  \gamma\left( \left(G|_U\right)^{-1}(w)\right) - w \in W^\perp, \quad w\in G(U)\subset W,$$
		  $\varphi: G(U)\subset W \to W^\perp$ is   $C^{1,1}$ smooth  and
		$$\graph\varphi:=\{w+\varphi(w):\, w\in G(U)\}=  \{\gamma\left(\left(G|_U\right)^{-1}(w)\right):\ w\in G(U)\}=\gamma(U).$$
		 Choose $0<\delta<\omega$ such that $(-\delta,\delta)\subset G(U)$. Then, by Fact~\ref{Wh2}, $\varphi|_{(-\delta,\delta)}:(-\delta,\delta)\to W^\perp$ has an $C^{1,1}$ extension $\tilde{\varphi}:W\to W^\perp$, $\Gamma:=\{w+\tilde{\varphi}(w):\, w\in W\}$  is a $1$-dimensional $C^{1,1}$ g-surface and we have 
       $$\Gamma\cap B(0,\delta) \subset \graph\varphi = \gamma(U)\subset D\subset A,$$
		which proves the second inclusion of \eqref{intj}.
    Further, if $z\in T_1(D)\cap B(0,\delta)$ then $\pi(z)\in(-\delta,\delta)$ and, hence,
    $$z':=\pi(z)+\varphi(\pi(z))\in \graph\varphi\cap \Gamma = \gamma(U)\cap\Gamma\subset D\cap\Gamma,$$
    hence $z'=z$ by \eqref{prnad}, which implies $z\in\Gamma$. This proves 
    $$T_1(D)\cap B(0,\delta)\subset\Gamma\cap B(0,\delta).$$
    Since $A\cap B(0,\delta) = D\cap B(0,\delta)$ and the notion of the tangent cone is local, we have  $T_1(A)\cap B(0,\delta)=T_1(D)\cap B(0,\delta)$. So the first inclusion of \eqref{intj} follows and the proof is complete.
\end{proof}

\section{Local structure of a $2$-dimensional set $A$ at points of $T_1(A)$}

In this section we prove our main result, Theorem~\ref{T_main}. As the proof is rather long, we divide the presentation into several subsections.

\subsection{Description of the strategy and some auxiliary notions} \label{ss51}

The difficulty of Theorem~\ref{T_main} consists in the proof of its first part (the proof of the second part is easy). It is advantageous to reformulate and prove the first part using an equivalence relation $\approx$ from the following definition.

\begin{definition}  \label{sim}
     Let $A,B\subset\R^d$,  $a\in A$ and $b \in B$. We write
     \begin{enumerate}
         \item[(i)] $A\sim_a B$ if $a\in A\cap B$ and $A\cap B(a,\ep)=B\cap B(a,\ep)$ for some $\ep>0$,
         \item[(ii)] $(A,a)\approx (B,b)$ if there exist an open set $U\ni a$ and a $C^{1,1}$-diffeo\-mor\-phism $\Phi:U\to\R^d$  such that $\Phi(a)=b$ and $\Phi(A\cap U) =B\cap \Phi(U)$.
     \end{enumerate}
\end{definition}

\begin{remark}\label{oekv1}
    \begin{enumerate}
        \item[(i)]  Clearly   $A\sim_a B$ implies $(A,a)\approx (B,a)$.  Further,
				 $(A,a)\approx (B,b)$ if and only there exist an open set $U\ni a$ and a $C^{1,1}$-diffeo\-mor\-phism $\Phi:U\to\R^d$  such that $\Phi(a)=b$ and $\Phi(A\cap U) \sim_b B$. The ``only if part'' is almost obvious.
				To prove the  ``if part'', suppose that $U$ and $\Phi$ as above are given and choose $\ep >0$
				 such that  $\Phi(A\cap U) \cap B(b,\ep) = B \cap B(b,\ep)$. Then, setting 
				$\tilde U:= \Phi^{-1}( B(b,\ep))$ and $\tilde \Phi:= \Phi |_{\tilde U}$, it is easy to check
				 that   $\tilde \Phi(A\cap \tilde U) =B\cap \tilde \Phi(\tilde U)$.
				
				\item[(ii)] 
			In the usual ``germ terminology'',  $A\sim_a B$  means that the set germs $[A]_a$ and $[B]_a$ are equal
				and  $(A,a)\approx (B,b)$ means that the set germs $[A]_a$ and $[B]_b$ are ``$C^{1,1}$-equivalent''.
		
        \item[(iii)] If $A$, $B$, $a$, $b$, $U$ and $\Phi$ are as in Definition \ref{sim} (ii) and $\tilde U$ is open with $a\in \tilde U \subset U$, then clearly     $\Phi(A\cap \tilde{U}) =B\cap \Phi(\tilde{U})$. This easily shows (cf.\ Remark~\ref{stability}~(a)) that  we can
 additionally postulate  in Definition \ref{sim} (ii) that $\Phi$ is bi-Lipschitz and $U$ is a ball.
\item[(iv)] 
		Not only $\sim$ but also $\approx$ is an equivalence relation. To prove symmetry, suppose
		that $(A,a)\approx (B,b)$ and  $U$ and $\Phi$ are as in Definition \ref{sim} (ii). Then
			 $\Phi^{-1} (B\cap \Phi(U)) = A\cap \Phi^{-1}(\Phi(U))$ shows that $(B,b) \approx (A,a)$.
			To prove transitivity, suppose that $(A,a)\approx (B,b)$ and $(B,b) \approx (C,c)$. Choose
			  $C^{1,1}$-diffeo\-mor\-phisms $\Phi:U\to\R^d$ and $\Psi:V\to\R^d$   such that $\Phi(a)=b$, 
				 $\Phi(A\cap U) =B\cap \Phi(U)$,  $\Psi(b)=c$ and $\Psi(B\cap V) =C\cap \Psi(V)$.
				 Putting  $\tilde U:= U \cap \Phi^{-1}(V)$ and  $\Omega:= \Psi \circ \Phi |_{\tilde U}$, we 
		easily check that  $\Omega(\tilde U \cap A) = C \cap \Omega(\tilde U)$ and $\Omega$ is a $C^1$-diffeomorphism such that $\Omega$ and $\Omega^{-1}$ are locally $C^{1,1}$ (see Remark~\ref{stability}~(c)).
        Choosing $W:=B(a,\delta)$ with $\delta>0$ sufficiently small, we obtain that the restriction $\Omega|_W$ is a $C^{1,1}$-diffeomorphism and $\Omega(A\cap W)=C\cap\Omega(W)$, hence,
         $(A,a)\approx (C,c)$ follows.
		\end{enumerate}
\end{remark}

We observe that the following proposition is a reformulation of the first part of Theorem~\ref{T_main}. Recall that multirotations were defined in Definition~\ref{multrot}.

\begin{proposition}  \label{T}
    Let $A\subset\R^d$ ($d\geq 3$) have positive reach, $\dim A\leq 2$ and $a\in T_1(A)$. Then there exists a $\B$-set  $B\subset\R^2\subset\R^d$ and a multirotation $\rho$ associated with $T_1(B)$ such that $(A,a)\approx(\rho(B),0)$.
\end{proposition}	

To prove this observation, it is clearly sufficient to prove that, for an {\it arbitrary} 	$A\subset\R^d$ ($d\geq 3$),
 the following properties are  equivalent:

\begin{enumerate}
\item[(i)]   There exist a $\B$-set $B_1\subset\R^2\subset\R^d$, a multirotation $\rho_1:\R^d\to\R^d$ associated with $T_1(B_1)$ and a $C^{1,1}$-diffeomorphism $\Phi:U\to\R^d$, where $U\subset\R^d$, $0\in U$, $a=\Phi(0)$, such that $\rho_1(B_1)\subset U$ and $\Phi(\rho_1(B_1))$ is a neighbourhood of $a$ in $A$. 

\item[(ii)]  There exists a $\B$-set  $B_2\subset\R^2\subset\R^d$ and a multirotation $\rho_2$ associated with $T_1(B_2)$ such that $(A,a)\approx(\rho_2(B_2),0)$.
\end{enumerate}
		
First, suppose that (i) holds. We have clearly  $\Phi(\rho_1(B_1)) \sim_a A$; thus 
$	(\rho_1(B_1),0) \approx	(A,a)$ by Remark \ref{oekv1} (i). Consequently  $(A,a) \approx		(\rho_1(B_1),0) $ 
 and so (ii) holds (with $B_2:= B_1$ and $\rho_2:= \rho_1$).   		
				
Second,   suppose that (ii) holds. Then $(\rho_2(B_2),0)  \approx (A,a)$ and so, by Remark~\ref{oekv1}~(iii),  there exist
$\delta>0$ and a  $C^{1,1}$-diffeomorphism $\Phi: B(0,\delta)\to\R^d$ such that  $\Phi(0)=a$ and
$$ \Phi(\rho_2(B_2) \cap B(0,\delta)) = A \cap \Phi(B(0,\delta)).$$
The definition of $B$-sets easily implies
 that there exists $\omega > 0$ such that  
$$  B_1:= \pi^{-1}([-\omega, \omega]) \cap B_2 \subset  B(0,\delta)$$
and consequently also $\rho_2(B_1) \subset B(0,\delta)$ (note that $|\rho_2(x)|=|x|$ for any $x\in\R^d$, by the definition of a multirotation). Obviously, $B_1$ is a $B$-set
 and it is easy to see that there exists a multirotation $\rho_1$ associated with $T_1(B_1)$ 
 such that $\rho_2(B_1) =  \rho_1(B_1)$ (cf.\ Lemma~\ref{mrot} below). Since  $\Phi(\rho_1(B_1))$ is clearly a neighbourhood of $a$ in $A$,
 (i) follows.

\begin{lemma}  \label{mrot}
    Let $C,C'\subset W:=\spa\{e_1\}=\R\subset\R^d$ ($d\geq 3$) be nonempty compact sets such that $C\cap[p,q]=C'\cap[p,q]$ for some given $p<q$. Let $\rho$ be a multirotation associated with $C$. Then there exists a multirotation $\rho'$ associated with $C'$ such that $\rho|_{\pi^{-1}[p,q]}=\rho'|_{\pi^{-1}[p,q]}$.
\end{lemma}

\begin{proof}
    Let $\rho$ be determined by $2$-rotations ($R_\lambda$), where $\lambda$ are components of $W\setminus C$. We define $\rho'$ as the multirotation determined by $2$-rotations $(R'_{\lambda'})$, where $\lambda'$ are components of $W\setminus C'$. We set $R'_{\lambda'}:=\id$ if $\lambda'\cap[p,q]=\emptyset$ and $R'_{\lambda'}:=R_\lambda$ if $\lambda'\cap[p,q]\neq\emptyset$, where $\lambda$ is the (clearly unique) component of $W\setminus C$ such that $\lambda\cap[p,q]=\lambda'\cap[p,q]$. It is easy to verify that $\rho'$ has the desired properties.
\end{proof}

\smallskip

To describe how we prove Proposition \ref{T} we need also the terminology introduced in Definition \ref{skew}.

\begin{notation}  \label{notation}
Throughout the section we will use the usual identifications $\R^d = \R \times \R^{d-1}$ and $\R=\R \times \{0\}=:W \subset \R^d$. Further, we denote  $\pi: = \pi_W: \R^d \to W$.
\end{notation}

\begin{definition}  \label{skew}
    We will say that a set $A\subset\R^d$ ($d\geq 2$) with positive reach is \emph{skewered on a segment} $[p,q]\subset\R$ if
    \begin{equation}\label{roz}
     \emptyset \neq T_1(A) \subset [p,q] \subset A  \subset \pi^{-1}([p,q]).
    \end{equation} 
    If $\lambda$ is a component of $\R\setminus T_1(A)$ with $\lambda\cap [p,q]\neq\emptyset$, we denote $A_\lambda:=A\cap\pi^{-1}(\lambda)$
    and call $A_\lambda$ a \emph{leave of $A$}. (Recall that $T_1(A)$ is closed by Lemma~\ref{T_k}~(iii).) 
    
    We say that $A\subset\R^d$ is a \emph{set skewered on $[p,q]$ with planar leaves} if $A$ is skewered on $[p,q]$ and each leave $A_\lambda$ of $A$ lies in some plane $V_\lambda\in G(d,2)$.
\end{definition} 

\begin{remark}   \label{Rem_listky}
    \begin{enumerate}
        \item[(a)] Note that each leave $A_\lambda$ is nonempty since it contains $\lambda\cap[p,q]$, and that $A_\lambda\cap T_1(A)=\emptyset$.
        \item[(b)] Any $\B$-set $B\subset\R^2\subset\R^d$ is skewered (on $[0,r]$ or $[-r,r]$ for some $r>0$), both in $\R^2$ and in $\R^d$. Indeed, $B$ has positive reach in $\R^2$ by Lemma~\ref{B_set_PR}, hence also in $\R^d$ by Remark~\ref{rem_reach}, and it is easy to see that $0\in T_1(B)$. Moreover, if $\varphi,\psi$ are the functions from Definition~\ref{B-sets}, we have
        $$T_1(B)\subset \{(x,0):\, \varphi(x)=\psi(x)\}\subset W,$$
        since for any point $(x,y)\in B$ with $\psi(x)<\varphi(x)$ it follows easily that $\dim\Tan(A,(x,y))=2$ from the Lipschitzness of $\varphi$ and $\psi$.

        (Conversely, using the convexity of tangent cones of $B$, it can be easily shown that any point $(x,0)\in B$ with $\varphi(x)=\psi(x)$ and $|x|\neq r$ belongs to $T_1(B)$.)
        
        \item[(c)] The set $M$ from \eqref{M} is a more sophisticated example of a skewered set in $\R^3$  with planar leaves. The system of leaves can be empty, finite or infinite here (as well as in the case of $\B$-sets).
        \end{enumerate}
\end{remark}

The proof of Proposition~\ref{T} for the case $a\in T_1^+(A)$ is divided into the proofs of the following three (mutually independent) propositions:

\begin{proposition}  \label{T1}
    Let $A\subset\R^d$ ($d\geq 3$) have positive reach, $\dim A\leq 2$ and $a\in T_1^+(A)$. Then there exists a set $S\subset\R^d$ with positive reach and $\dim S\leq 2$  skewered on a segment $[c,d]$, $c<0<d$, such that $0\in T_1(S)$ and $(A,a)\approx(S,0)$.
\end{proposition}

The main ingredient of the proof of Proposition~\ref{T1} is Theorem~\ref{T1+}.

\begin{proposition}  \label{T2}
    Let $S\subset\R^d$ ($d\geq 3$) with positive reach and $\dim S\leq 2$ be skewered on $[c,d]$, $c<0<d$, and $0\in T_1(S)$. Then there exists a set $\tilde{S}\subset\R^d$ with positive reach skewered on $[\tilde{c},\tilde{d}]$, $\tilde{c}<0<\tilde{d}$, with planar leaves and such that $0\in T_1(\tilde{S})$ and $(S,0)\approx(\tilde{S},0)$.
\end{proposition}

The proof of this proposition is the most difficult part of the proof of Theorem~\ref{T_main}. It uses the $C^{1,1}$ Whitney theorem and some facts of the theory of product integration that will be explained before the proof.

\begin{proposition}  \label{T3}
    Let $S\subset\R^d$ ($d\geq 3$) with positive reach be a set skewered on $[c,d]$, $c<0<d$, with planar leaves, and $0\in T_1(S)$. Then there exists a $\B_+$-set $B\subset\R^2\subset\R^d$ and a multirotation $\rho$ associated with $T_1(B)$ such that $S\sim_0\rho(B)$.
\end{proposition}

Using the transitivity of relation $\approx$, we obtain Proposition~\ref{T} in the case $a\in T_1^+(A)$ as an immediate consequence of Propositions~\ref{T1}, \ref{T2} and \ref{T3}. The case $a\in T_1^-(A)$ is then inferred from the case $a\in T_1^+(A)$ in Subsection~\ref{subs_last}.

\subsection{Proof of Proposition~\ref{T1}}

The main step of the proof is an application of Theorem~\ref{tdmnapl2}. We will also need the following easy lemma.

\begin{lemma}  \label{neprib}
    Let $A\subset \R^d$, $0<\tau < r < \reach A$, $b\in A$ and $A^*:= A \cap \overline B(b, \tau)$. Then $\reach A^*>r$ and $T_k(A^*) = T_k(A) \cap A^*$, $k=0,\dots,d$.
\end{lemma}

\begin{proof} The first assertion follows from Lemma~\ref{ball} (applied with an $r^*\in(r,\reach A)$).
    For the second assertion it is sufficient to prove that
    \begin{equation}  \label{samedim}
        \dim\Tan(A,a)=\dim\Tan(A^*,a),\quad a\in A^*.
    \end{equation}
    If $a\in B(b,\tau)$ then  \eqref{samedim} is obvious. Assume that $a\in\partial B(b,\tau)$ and denote $k:=\dim\Tan(A,a)$ and $H:=\{x\in\R^d:\,\langle x,b-a\rangle>0\}$. 
    By Corollary~\ref{uhlovy} there exists $v\in \Tan(A,a)\cap S^{d-1}$ such that $\angle(v, b-a) < \pi/6$, hence $v\in H$ and thus $\Tan(A,a)\cap H\neq\emptyset$. Then also the relative interior of $\Tan(A,a)$ meets $H$ by \cite[Corollary~6.3.2]{RT70}, hence $\Tan(A,a)\cap H$ contains a $k$-dimensional ball.    
    It follows easily from the definition of the tangent cone that $\Tan(A,a)\cap H\subset\Tan(A^*,a)$, hence $\dim\Tan(A^*,a)\geq k$. As the other inequality is obvious, we have $\dim\Tan(A^*,a)=k$ and the proof is finished.
\end{proof}
    
\begin{proof}[Proof of Proposition~\ref{T1}]  \vskip 2mm
Choose  $0<r < \reach A$ and  an affine isometry $\Omega: \R^d \to \R^d$ such that
$\Omega(a)=0$ and $\Tan(\Omega(A),0)=W$.
Clearly, $E:=\Omega(A)$ fulfills  $r< \reach E$ and $\dim E \leq 2$. 

Applying Theorem~\ref{T1+} we find that there exists  a $1$-dimensional $C^{1,1}$ surface $\Gamma$ and $\delta>0$ such that
\begin{equation}\label{inkluze2}
  T_1(E) \cap B(0,\delta) \subset \Gamma \cap B(0,\delta) \subset E.
	\end{equation}
Moreover, inspecting the proof of Theorem~\ref{T1+}, we easily see that we can choose
$$  \Gamma= \{ w+ \vf(w):\ w\in W\}$$
with a suitable $C^{1,1}$ mapping $\vf:W \to W^\perp$.
 Now define the mapping
\begin{equation}\label{Phi}
	\Phi(z) = z - \vf(\pi(z)),\ \ \ z \in \R^d.
\end{equation}
It is easy to see that $\Phi: \R^d \to \R^d$ is  a $C^{1,1}$ mapping which is even  $C^{1,1}$-diffeomorphism onto $\R^d$, since $\Phi^{-1}(z) = z + \vf(\pi(z))$, $z\in\R^d$.
Clearly  $\Phi(\Gamma)= W$, $\Phi^{-1}(W)=\Gamma$, $\Phi(0)=0$ and $\reach(\Phi(E))>0$ by Lemma~\ref{reach_C11}.
Further, choose $0<\omega<\reach\Phi(E)$ such that 
\begin{equation}\label{omga}
 \overline B(0,\omega) \subset \Phi(B(0,\delta)),
\end{equation}
and set $S:= \Phi(E)\cap\overline{B}(0,\omega)$ and $[c,d]:=[-\omega,\omega]$. We have $\reach S>0$ by Lemma~\ref{ball}.
We will show that 
\begin{equation}\label{vlm}
S\  \text{ is skewered on the segment }\  [c,d].
\end{equation}
To this end, first note that obviously $S \subset \pi^{-1}([c,d])$ and we have $0 \in T_1(S)$ by Lemma~\ref{tkdif}~(iii).
Further, consider an arbitrary $w\in [c,d] \subset W$ and set $z:= \Phi^{-1}(w) \in \Gamma$.
Then \eqref{omga} implies $z \in B(0,\delta)$ and so $z \in E$ by \eqref{inkluze2}. Since
$w= \Phi(z)$ we have $w \in S$. Thus $[c,d] \subset S$.
Finally, let  $x\in T_1(S)$ be given.  Using Lemma~\ref{neprib} and Lemma~\ref{tkdif}~(iii) we obtain
$x\in T_1(\Phi(E)) = \Phi(T_1(E))$. So $x= \Phi(z)$ for some $z\in T_1(E)$. Since $z= \Phi^{-1}(x)$ and $x\in S \subset \overline B(0,\omega)$, we obtain $z\in B(0,\delta)$ by \eqref{omga}.
Thus $z\in \Gamma$ by \eqref{inkluze2} and consequently  $x= \Phi(z)\in W \cap \overline B(0,\omega) =  [c,d]$. This completes the proof of \eqref{vlm}.
		
Further, by our construction, we have
$$ (A,a) \approx (\Omega(A),0)  \approx (E,0) \approx (\Phi(E),0) \approx (S,0),$$
hence $(A,a) \approx  (S,0)$, and recall that $0\in T_1(S)$. Thus the proof is complete.
\end{proof}

\subsection{Some properties of skewered sets with positive reach}

The following two lemmas will be used in the proof of the most difficult Proposition~\ref{T2}. Recall that Notation~\ref{notation} is used.

\begin{lemma}  \label{skew-prop}
    Let $A\subset\R^d$ ($d\geq 2$) be a set with positive reach with $\dim A\leq 2$ skewered on a segment $[p,q]\subset\R=W$, $0<r<\reach A$ and $x_0\in T_1(A)$. Then the following statements hold.
    \begin{enumerate}
        \item[{\rm (i)}] If $\diam A<r$ then $T_0(A)=\emptyset$.
        \item[{\rm (ii)}] If $0<\rho<r$ then $A\cap\overline{B}(x_0,\rho)$ is a set with positive reach of dimension at most $2$ with $\reach(A\cap\overline{B}(x_0,\rho))>r$ skewered on $[p,q]\cap [x_0-\rho,x_0+\rho]$. 
        \item[{\rm (iii)}] If $x\in A$ and $|x-x_0|<r$ then
        $$|\pi(x)-x_0|\geq\frac{\sqrt{3}}{2}|x-x_0|\geq\sqrt{3}|x-\pi(x)|.$$ 
    \end{enumerate}
    Assume further that
    $$\diam A < r\quad\text{ or }\quad A\subset \R^2 \subset \R^d \text{ is a }B\text{-set.}$$
    Then the following statements hold.
    \begin{enumerate}
    \item[{\rm (iv)}] If $z\in T_1(A)$ then $A\cap\pi^{-1}(\{z\})=\{z\}$. Equivalently,  $A\setminus T_1(A)$ is equal to the union of all leaves of $A$.
    \item[{\rm (v)}] A leave $A_\lambda$ with $\lambda=(a_\lambda,b_\lambda)$ can be of three types: 
    \begin{itemize}
            \item if $a_\lambda, b_\lambda\in [p,q]$ then $a_\lambda,b_\lambda\in T_1(A)$, $\pi(A_\lambda)=\lambda$ and $\overline{A}_\lambda\setminus A_\lambda=\{a_\lambda,b_\lambda\}$, 
            \item if $a_\lambda=-\infty$ and $b_\lambda\in (p,q]$ then $b_\lambda\in T_1(A)$, $\pi(A_\lambda)=[p,b_\lambda)$ and $\overline{A}_\lambda\setminus A_\lambda=\{b_\lambda\}$, 
            \item if $a_\lambda\in [p,q)$ and $b_\lambda=\infty$ then $a_\lambda\in T_1(A)$, $\pi(A_\lambda)=(a_\lambda,q]$ and $\overline{A}_\lambda\setminus A_\lambda=\{a_\lambda\}$.
    \end{itemize}
    \item[{\rm (vi)}] For each leave $A_\lambda$ we have $\reach\overline{A}_\lambda>r$.
    \end{enumerate}
\end{lemma}

\begin{proof}
If $\diam A<r$ then $A$ is connected by Lemma~\ref{P_PR}~(v), and since it contains the segment $[p,q]$, it cannot contain isolated points. Hence, $T_0(A)=\emptyset$ by Lemma~\ref{T_k}~(iv) and assertion (i) follows.    Statement (ii) follows immediately from Lemma~\ref{neprib}, and  (iii) is an application of Corollary~\ref{uhlovy},  \eqref{tan_proj}, since clearly $x_0+\spa(\Tan(A,x_0))=W=\R$. 

Assertion (iv) follows from (iii) in the case $\diam A<r$, and from Remark~\ref{Rem_listky}~(b) in the case of a $B$-set $A$. Property (v) follows from (iv) and the closedness of $A$.

Finally, we prove (vi) using Proposition~\ref{fedtan}. Let $r<r^*<\reach A$ and $a,b\in\overline{A}_\lambda$, $\lambda=(a_\lambda,b_\lambda)$, be given. If $a\in A_\lambda$ then $\Tan(A,a)=\Tan(\overline{A}_\lambda,a)$ (since $A_\lambda$ is relatively open in $A$ and the tangent cone is a local notion). Thus we have
\begin{equation} \label{Fedab}
\dist(b-a,\Tan(\overline{A}_\lambda,a))=\dist(b-a,\Tan(A,a))\leq\frac{|b-a|^2}{2r^*}.
\end{equation}
If, on the other hand, $a\in\overline{A}_\lambda\setminus A_\lambda$, we have $a\in T_1(A)\cap\{a_\lambda,b_\lambda\}$ by (v). Then the definition of a leave easily gives
$$\dist(b-a,\Tan(\overline{A}_\lambda,a))=\dist(b-a,W)=\dist(b-a,\Tan(A,a)),$$
hence \eqref{Fedab} holds again and we infer that $\reach \overline{A}_\lambda\geq r^*>r$ using Proposition~\ref{fedtan}.
\end{proof}

\begin{lemma}  \label{skewer}
Assume that an at most two-dimensional set $A\subset\R^d$ ($d\geq 2$) with positive reach is skewered on a segment $[p,q]\subset\R$, $\reach A> r>0$ and $\diam A<r$. Then, for each leave $A_\lambda$ of $A$, $\reach\overline{A}_\lambda> r$, $A_\lambda\subset T_2(A)$ and $\psi^A_2$ is $(2^{12}\pi/r)$-Lipschitz on $A_\lambda$.
\end{lemma}

\begin{proof}
    Choose a leave $A_\lambda$ with $\lambda=(a_\lambda,b_\lambda)$. 
    We know already that $\reach\overline{A}_\lambda> r$ from Lemma~\ref{skew-prop}~(vi). 
    
    We have $A_\lambda\cap T_1(A)=\emptyset$ (cf.\ Remark~\ref{Rem_listky}~(a)) and, since also $T_0(A)=\emptyset$ by  Lemma~\ref{skew-prop}~(i), using \eqref{tan_k} we obtain $A_\lambda\subset T_2(A)$. 
    
    Since $A_\lambda$ is relatively open in $A$, it follows that $\Tan(A,x)=\Tan(\overline{A}_\lambda,x)$ and so $\psi_2^A(x)=\psi_2^{\overline{A}_\lambda}(x)$ whenever $x\in A_\lambda$. Hence $A_\lambda\subset T_2(\overline{A}_\lambda)$. 
    We obtain from Lemma~\ref{skew-prop}~(v) and the definition of leaves that
    $\overline{A}_\lambda\setminus A_\lambda\subset T_1(A)\cap\{a_\lambda,b_\lambda\}\subset T_1^-(\overline{A}_\lambda)$.
    Thus we can apply Lemma~\ref{lipna2} to $\overline{A}_\lambda$ and obtain that $\psi_2^A|_{A_\lambda}=\psi_2^{\overline{A}_\lambda}$ is $(2^{12}\pi/r)$-Lipschitz.
\end{proof}

\subsection{Product-integrals of operators and the proof of Proposition~\ref{T2}}
 
We will use some basic facts from the well-known theory of \emph{product-integrals} (cf.\ \cite{GJ90}), called also \emph{continuous products} (cf.\ \cite{MacN}). We will use the exposition from \cite{GJ90}.

Under an \emph{interval function} we understand a function of half-closed intervals $(s,t]\subset[0,\infty)$. We will consider square $(n\times n)$ matrix-valued interval functions. For such a function $\alpha$ we will write $\alpha(s,t)$ instead of $\alpha((s,t])$ for short.
We call an interval function $\alpha$ \emph{additive} if $\alpha(s,u)=\alpha(s,t)+\alpha(t,u)$, $0\leq s\leq t\leq u$, where we set $\alpha(v,v):=0$, $v\geq0$, and \emph{right-continuous} if $\lim_{t\to s_+}\alpha(s,t)=0$, $s\geq 0$.
We say that a matrix-valued additive interval function $\alpha(s,t)$ is \emph{dominated} by a scalar additive interval function $\alpha_0(s,t)$ if $\|\alpha(s,t)\|_{GJ}\leq\alpha_0(s,t)$, $0\leq s\leq t$ where 
$$\left\|\left(a_{ij}\right)_{i,j=1}^n\right\|_{GJ}:=\max_i\sum_j|a_{ij}|.$$

Let a matrix-valued additive interval function $\alpha(s,t)$ be dominated by a scalar additive and right-continuous interval function $\alpha_0(s,t)$. 
Then, clearly, $\alpha$ is right-continuous as well. In this situation, by \cite[Theorem~1]{GJ90}, we can define the product-integral 
$$\mu(s,t)=\mu((s,t])=\prod_{(s,t]}(I+d\alpha):=\lim_{\|\cD\|\to 0}\prod_{i=1}^k(I+\alpha(t_{i-1},t_i)),\quad 0\leq s<t,$$
where $I$ denotes the identity matrix and $\cD=\{s=t_0<t_1<\dots<t_k=t\}$ is a partition of $[s,t]$ with mesh $\|\cD\|$. We also set $\mu(s,s):=I$, $s\geq 0$. Then $\mu(s,t)$ is \emph{multiplicative} in the sense that 
\begin{equation} \label{multiplicative}
\mu(s,u)=\mu(s,t)\cdot \mu(t,u), \quad 0\leq s\leq t\leq u,
\end{equation}
and $\mu-I$ is dominated by $\mu_0-1$ (see \cite[Theorem~1]{GJ90}), where
$$\mu_0(s,t):=\prod_{(s,t]}(1+d\alpha_0)$$
(which exists by \cite[Proposition~3]{GJ90}).
Also, by \cite[Proposition~1]{GJ90} (see also \cite[(20)]{GJ90}), we have
$$\mu_0(s,t)=\prod_{(s,t]}(1+d\alpha_0)\leq\exp(\alpha_0(s,t)).$$
Consequently, we have
\begin{equation}  \label{mu_dom}
\|\mu(s,t)-I\|_{GJ}\leq \exp(\alpha_0(s,t))-1,\quad 0\leq s\leq t.
\end{equation}

\begin{remark}  \label{rem_GJ}
    (i) An additive interval function $\alpha$ can be identified with the distribution function $F(x)=\alpha(0,x)$, $x\geq 0$, which is used in \cite{MacN} to define the `continuous product' $\prod_s^t(I+dF)$. The approach of \cite{MacN} is slightly more general, in particular, $F$ is not assumed to be right-continuous, but only of finite variation. Note that in \cite{GJ90}, right-continuity is already included in the definition of an additive interval function, which we do not adopt here (but our interval functions are right continuous).

    (ii) In our application we will work with interval functions valued in $\cL(\R^d,\R^d)$ (linear mappings from $\R^d$ to $\R^d$), which we identify with the space of all $d\times d$ matrices. We use, however, the usual operator norm $\|\cdot\|$ on  $\cL(\R^d,\R^d)$ which differs from $\|\cdot\|_{GJ}$. We will use the well-known fact that we can choose a constant $k_d\geq 1$ such that
    \begin{equation}  \label{norm-eq}
    k_d^{-1}\|\cdot\|_{GJ}\leq\|\cdot\|\leq k_d\|\cdot\|_{GJ}.
    \end{equation}
\end{remark}

\begin{proof}[Proof of Proposition~\ref{T2}] \vskip 2mm
    Recall that Notation~\ref{notation} is used.
    Let $r>0$ be such that $\reach S> r>0$ and denote $L:=2^{12}\pi/r$. Choose $\rho>0$ such that $[-\rho,\rho]\subset[c,d]$ and 
    \begin{equation} \label{rho} 
    2\rho<1/(Lek_d^2)<r,
    \end{equation}
    and set $S':=S\cap\overline{B}(0,\rho)$. Then $\reach S'>r$ and $S'$ is skewered on $[-\rho,\rho]$ (we use  Lemma~\ref{skew-prop}~(ii)). Set $A:=S'+\rho$. (Recall that $\rho\in \R= W\subset\R^d$ due to Notation~\ref{notation}.) Then
    \begin{equation}  \label{diam<}
    \reach A>r,\quad \diam A\leq 2\rho<r,\quad \dim A\leq 2, \quad \rho\in T_1(A)
    \end{equation}
     and $A$ is skewered on $[0,2\rho]$.

    Denote $C:=T_1(A)\subset [0,2\rho]$ and let $\Lambda$ be the family of all components $\lambda$ of $W\setminus C$ for which $\lambda\cap[0,2\rho]$ is nonempty. Denote
    $$(a_\lambda,b_\lambda):=\lambda\cap(0,2\rho),\quad \lambda\in\Lambda.$$
    Lemma~\ref{skewer} together with \eqref{diam<} imply that for each $\lambda\in\Lambda$, the leave $A_\lambda=A\cap\pi^{-1}(\lambda)$ satisfies $\reach \overline{A}_\lambda>r$ and the mapping $\psi_2^A: x\mapsto \widetilde{\Tan}(A,x)$ is $L$-Lipschitz on $A_\lambda\subset T_2(A)$, hence, $\psi_2^A|_{A_\lambda}$ has an $L$-Lipschitz extension to $\psi^\lambda:\overline{A}_\lambda\to G(d,2)$. (Note that $G(d,2)$ is complete, cf.\ our remark after Definition~\ref{gap}.)

    Denote $T_\lambda:=\psi^\lambda(a_\lambda)$, $\lambda\in\Lambda$. We claim that
    \begin{equation}   \label{Tan=W}
W\subset T_\lambda=\psi^\lambda(a_\lambda)\quad\text{ and }\quad W\subset\psi^\lambda(b_\lambda).
    \end{equation}
    Indeed, since $T_\lambda=\lim_{t\to 0_+}\psi_2^A(a_\lambda+t)$ and $\psi_2^A(a_\lambda+t)$ contains $W$ for all $t>0$ small enough, it is easy to see that also $T_\lambda$ contains $W$. The second inclusion is obtained quite analogously.
    By the Lipschitzness of $\psi^\lambda$ we have
    \begin{equation} \label{lip_psi}
    \rho_2(T_\lambda,\psi^\lambda(x))\leq L|x-a_\lambda|\leq L|b_\lambda-a_\lambda|,\quad x\in\overline{A}_\lambda.
    \end{equation}
    Consequently, using \eqref{vlga}, we get
    \begin{equation}  \label{vzd_pr}
        \left\|\pi_{\psi^\lambda(a_\lambda)}-\pi_{\psi^\lambda(b_\lambda)}\right\|=\rho_2(\psi^\lambda(a_\lambda),\psi^\lambda(b_\lambda))\leq L|b_\lambda-a_\lambda|.
    \end{equation}

    Consider the operator-valued interval function
    $$\alpha(s,t):=\sum_{s< b_\lambda\leq t}\left(\pi_{\psi^\lambda(a_\lambda)}-\pi_{\psi^\lambda(b_\lambda)}\right),\quad 0\leq s\leq t$$
    (an empty sum is defined as $0$).
    The sum converges (absolutely) since, due to \eqref{vzd_pr},
    $$\sum_{s< b_\lambda\leq t}\left\|\pi_{\psi^\lambda(a_\lambda)}-\pi_{\psi^\lambda(b_\lambda)}\right\|\leq \omega(s,t):=L\sum_{s< b_\lambda\leq t}|b_\lambda-a_\lambda|\leq Lt.$$
    Set $\alpha_0:=k_d\omega$. Clearly, both $\alpha$ and $\alpha_0$ are additive interval functions, $\alpha_0$ is right-continuous and 
    $\|\alpha(s,t)\|_{GJ}\leq k_d\|\alpha(s,t)\|\leq k_d\omega(s,t)=\alpha_0(s,t)$, hence, $\alpha$ is dominated by $\alpha_0$.    
    Hence, the product integral
    $$\mu(s,t):=\prod_{(s,t]}(I+d\alpha),\quad 0\leq s\leq t,$$
    exists, is multiplicative (cf.\ \eqref{multiplicative}) and we have by \eqref{mu_dom} 
    $$\|\mu(s,t)-I\|\leq k_d\|\mu(s,t)-I\|_{GJ}\leq k_d \left(\exp(\alpha_0(s,t))-1\right),\quad 0\leq s\leq t.$$
    Also, we easily obtain the estimate
    \begin{equation} \label{alfa0}
    \omega(s,t)\leq L|t-s|,\quad s\leq t,\quad s\in C':=C\cup\{0,2\rho\}.
    \end{equation}
    Consequently, by \eqref{alfa0}, for $2\rho\geq t\geq s\in C'$ we have
    \begin{equation}  \label{mu-I}
    \|\mu(s,t)-I\|\leq k_d\left(\exp(k_dL|t-s|)-1\right)\leq M|t-s|
    \end{equation}
    with $M:=k_d^2eL$, since $k_dL|s-t|\leq 2k_dL\rho\leq 1$ by \eqref{rho}, and $e^x-1\leq ex$ for $x\in[0,1]$. Since $2M\rho<1$ by \eqref{rho}, we obtain for all $2\rho\geq t\geq s\in C'$ that $\|\mu(s,t)-I\|<1$ and, consequently, $\|\mu(s,t)\|<2$ and $\mu(s,t)$ is a bijective linear mapping.
    So, denoting $\mu_t:=\mu(0,t)$, $t\geq 0$, 
    we have
    \begin{equation} \label{mu_bounded}
    \mu_t\text{ is bijective and }\|\mu_t\|\leq 2,\quad t\in [0,2\rho].
    \end{equation}
    Let us say that a linear mapping $G:\R^d\to\R^d$ satisfies \emph{property} (P) if $G(x)=x$ whenever $x\in W$, and $\pi\circ G=\pi$. 
    We shall show that
    \begin{equation}   \label{P}
        \pi_{T_\lambda}\text{ and }\mu_t\text{ have property (P)},\quad \lambda\in\Lambda,\, t\geq 0.
    \end{equation}
    Using \eqref{Tan=W}, it is easy to check that $\pi_{\psi^\lambda(a_\lambda)}=\pi_{T_\lambda}$ and $\pi_{\psi^\lambda(b_\lambda)}$ satisfy property (P).
     Thus $(\pi_{\psi^\lambda(a_\lambda)}-\pi_{\psi^\lambda(b_\lambda)})|_W=0$ 
    and 
    $\pi\circ(\pi_{\psi^\lambda(a_\lambda)}-\pi_{\psi^\lambda(b_\lambda)})=\pi-\pi=0$, hence for any $0\leq s\leq t$,
    $\alpha(s,t)|_W=\sum_{s<b_\lambda\leq t}(\pi_{\psi^\lambda(a_\lambda)}-\pi_{\psi^\lambda(b_\lambda)})|_W=0$ 
    and 
    $$\pi\circ\alpha(s,t)=\pi\circ\sum_{s<b_\lambda\leq t}(\pi_{\psi^\lambda(a_\lambda)}-\pi_{\psi^\lambda(b_\lambda)})=\sum_{s<b_\lambda\leq t}\pi\circ(\pi_{\psi^\lambda(a_\lambda)}-\pi_{\psi^\lambda(b_\lambda)})=0.$$
    Thus we get that $I+\alpha(s,t)$ satisfies property (P) for any $0\leq s\leq t$.
    Further, property (P) is preserved under finite compositions and limits of sequences of linear operators. Hence, using the definition of the product-integral, we obtain that $\mu_t=\mu(0,t)$ is the limit of a sequence of operators with property (P) and, consequently, has property (P) as well. 

    We define the linear mappings
    $$\varphi_\lambda(x):= \pi_{T_\lambda}+\pi_{\psi^\lambda(x)^\perp}\quad x\in\overline{A}_\lambda,\quad \lambda\in\Lambda.$$
    Since clearly 
    \begin{equation}  \label{45}
        \varphi_\lambda(x)-I=\pi_{T_\lambda}+\pi_{\psi^\lambda(x)^\perp}-
        \left(\pi_{\psi^\lambda(x)}+\pi_{\psi^\lambda(x)^\perp}\right)=\pi_{T_\lambda}-\pi_{\psi^\lambda(x)},
    \end{equation} 
    we obtain
    \begin{equation} \label{phiab}
    \varphi_\lambda(a_\lambda)=I\text{ and } \varphi_\lambda(b_\lambda)=I+\pi_{\psi^\lambda(a_\lambda)}-\pi_{\psi^\lambda(b_\lambda)},
    \end{equation}
    and, using \eqref{45}, \eqref{vlga} and \eqref{lip_psi}, 
    \begin{equation} \label{Q-I}
    \|\varphi_\lambda(x)-I\|\leq L |x-a_\lambda|,\quad x\in \overline{A}_\lambda.
    \end{equation}
    Note also that, by definition of the product-integral,
    \begin{equation}    \label{muab}
    \mu(a_\lambda,b_\lambda)=I+\pi_{\psi^\lambda(a_\lambda)}-\pi_{\psi^\lambda(b_\lambda)}=\varphi_\lambda(b_\lambda),
    \end{equation}
    since for each $(s,t)\subset(a_\lambda,b_\lambda)$, $\alpha(s,t)=\pi_{\psi^\lambda(a_\lambda)}-\pi_{\psi^\lambda(b_\lambda)}$ if $t=b_\lambda$ and $\alpha(s,t)=0$ otherwise.
    
    We define further functions $f:A\to\R^d$ and $\varphi:A\to\cL(\R^d,\R^d)$ as follows:
    \begin{align*}
        f(x)&:=\begin{cases}
            x, & x\in C,\\
            (\mu_{a_\lambda}\circ \pi_{T_\lambda})(x),& x\in \overline{A}_\lambda,\, \lambda\in\Lambda,
        \end{cases}\\
        \varphi(x)&:=\begin{cases}
            \mu_x, & x\in C,\\
            \mu_{a_\lambda}\circ \varphi_\lambda(x), & x\in \overline{A}_\lambda,\, \lambda\in\Lambda.
        \end{cases}
    \end{align*}
    First, we have to verify the consistency of the definitions in the case if $x\in C\cap\overline{A}_\lambda$ for some $\lambda\in\Lambda$ (then $x=a_\lambda$ or $x=b_\lambda$). Note that $\pi_{T_\lambda}(x)=\mu_{a_\lambda}(x)=x$ in both cases by \eqref{P}, hence, the definition of $f$ is consistent. Further, we have $\varphi_\lambda(a_\lambda)=I$ (see \eqref{phiab}) and $\mu_{a_\lambda}\circ \varphi_\lambda(b_\lambda)=\mu_{a_\lambda}\circ \mu(a_\lambda,b_\lambda)=\mu_{b_\lambda}$ by \eqref{muab} and the multiplicativity of $\mu$ (see \eqref{multiplicative}), hence, also the definition of $\varphi$ is consistent.

    Using the definition of $f$ and \eqref{P} we obtain
    \begin{equation} \label{Pf}
        f(x)=x,\ x\in A\cap W,\quad\text{ and }\quad\pi(f(x))=\pi(x),\quad x\in A.
    \end{equation}

    Now we are going to extend $f$ using Fact~\ref{Wh}. To this end, we will find $c>0$ such that conditions \eqref{Whitney_1} and \eqref{Whitney_2} hold for $f$, $\varphi$ and all $x,y\in A$. We will distinguish several cases.

    Case I: $x,y\in C$. To prove \eqref{Whitney_1}, we can assume without loss of generality that $x\leq y$ and, using the multiplicativity of $\mu$, we have
    $$\varphi(y)-\varphi(x)=\mu_y-\mu_x=\mu(0,y)-\mu(0,x)=\mu(0,x)(\mu(x,y)-I),$$
    and \eqref{Whitney_1} follows from \eqref{mu-I} and \eqref{mu_bounded} with $c_1:=2M$. Further, we have $f(y)-f(x)-\varphi(x)(y-x)=0$ since $\varphi(x)(y-x)=\mu_x(y-x)=y-x$ by \eqref{P}, hence \eqref{Whitney_2} is true with $c_1$ as well. 
        
    Case II: $x,y\in\overline{A}_\lambda$ for some $\lambda\in\Lambda$. Then
    \begin{align*}
        \|\varphi(y)-\varphi(x)\|&=\|\mu_{a_\lambda}\circ(\varphi_\lambda(y)-\varphi_\lambda(x))\|\\
        &= \|\mu_{a_\lambda}\circ(\pi_{\psi^\lambda(y)^\perp}-\pi_{\psi^\lambda(x)^\perp})\|\\
        &\leq \|\mu_{a_\lambda}\|\|\pi_{\psi^\lambda(y)^\perp}-\pi_{\psi^\lambda(x)^\perp}\|\\
        &= \|\mu_{a_\lambda}\|\rho_2(\psi^\lambda(x),\psi^\lambda(y))\leq 2L|y-x|,
    \end{align*}
    where \eqref{vlga} and the $L$-Lipschitzness of $\psi^\lambda$ were used. Also, 
     using \eqref{mu_bounded}, we get
    \begin{align*}
        |f(y)-&f(x)-\varphi(x)(y-x)|=|\mu_{a_\lambda}\left(\pi_{T_\lambda}(y)-\pi_{T_\lambda}(x)-\varphi_\lambda(x)(y-x)\right)|\\
        &=|\mu_{a_\lambda}\left(\pi_{\psi^\lambda(x)^\perp}(x-y)\right)|
        \leq 2|\pi_{\psi^\lambda(x)^\perp}(y-x)|\leq 2\frac{|y-x|^2}{2r},
    \end{align*}
    where the last estimate follows from Proposition~\ref{fedtan}, since 
    $$|\pi_{\psi^\lambda(x)^\perp}(y-x)|=\dist(y-x,\psi^\lambda(x))\leq \dist(y-x,\Tan(A,x))$$
    (we use here that $\Tan(A,x)\subset\psi^\lambda(x)$ which is obvious if $x\in A_\lambda$ and follows from \eqref{Tan=W} and the fact that $\Tan(A,x)\subset W$ if $x\in \overline{A_\lambda}\setminus A_\lambda\subset C$).
    Hence, \eqref{Whitney_1} and \eqref{Whitney_2} follow with $c_2:=\max\{2L,\frac{1}{r}\}$.

    Case III: $x\in\overline{A}_\beta$ and $y\in\overline{A}_\lambda$ with $\beta\neq\lambda$, $\beta,\lambda\in\Lambda$. Assume that $b_\beta\leq a_\lambda$ (the case $b_\lambda\leq a_\beta$ can be done quite analogously). 
    Since both $b_\beta$ and $a_\lambda$ belong to $T_1(A)$, $|x-b_\beta|<r$, $|y-a_\lambda|<r$ and $|x-a_\lambda|<r$, we have from Lemma~\ref{skew-prop}~(iii) 
    \begin{equation*} 
    |\pi(y)-a_\lambda|\geq\tfrac{\sqrt{3}}{2} |y-a_\lambda|, \quad |b_\beta-\pi(x)|\geq\tfrac{\sqrt{3}}{2}|b_\beta-x|,\quad |\pi(x)-a_\lambda|\geq\tfrac{\sqrt{3}}{2} |x-a_\lambda|,
    \end{equation*}
    which gives (together with the easy observation $|a_\lambda-b_\beta|\leq|x-y|$) 
    \begin{equation} \label{duvf}
        \max\left\{|y-a_\lambda|,|x-a_\lambda|,|b_\beta-x|,|a_\lambda-b_\beta|\right\}\leq\tfrac{2}{\sqrt{3}}|y-x|.
    \end{equation}
    Hence, using the estimates from Cases I and II, we have
    \begin{align*}
        \|\varphi(y)-\varphi(x)\|&\leq \|\varphi(y)-\varphi(a_\lambda)\|+\|\varphi(a_\lambda)-\varphi(b_\beta)\|+\|\varphi(b_\beta)-\varphi(x)\|\\
        &\leq c_2(|y-a_\lambda|+|b_\beta-x|)+c_1|a_\lambda-b_\beta|\\
        &\leq c_2\tfrac{2}{\sqrt{3}}(|\pi(y)-a_\lambda|+|b_\beta-\pi(x)|)+c_1|a_\lambda-b_\beta|\\
        &\leq c_3'|\pi(y)-\pi(x)|\leq c_3'|y-x|
    \end{align*}
    with $c_3':=\max\{c_2\frac{2}{\sqrt{3}},c_1\}$,   where we have used \eqref{duvf}. Also, using the estimates from Case II, we get
    \begin{align*}
        |f(y)-f(x)-&\varphi(x)(y-x)|\\
        \leq&|f(y)-f(a_\lambda)-\varphi(a_\lambda)(y-a_\lambda)|\\
        &+|f(a_\lambda)-f(b_\beta)-\varphi(b_\beta)(a_\lambda-b_\beta)|\\
        &+|f(b_\beta)-f(x)-\varphi(x)(b_\beta-x)|\\
        &+|(\varphi(a_\lambda)-\varphi(x))(y-a_\lambda)|+|(\varphi(b_\beta)-\varphi(x))(a_\lambda-b_\beta)|\\
        \leq& c_2(|y-a_\lambda|^2+|b_\beta-x|^2)\\ 
        &+\|\varphi(a_\lambda)-\varphi(x)\||y-a_\lambda|+  \|\varphi(b_\beta)-\varphi(x)\||a_\lambda-b_\beta|\\
        \leq& c_2(|y-a_\lambda|^2+|b_\beta-x|^2)\\
        &+c_2\left(|a_\lambda-x||y-a_\lambda|+|b_\beta-x||a_\lambda-b_\beta|\right)\\
        &\leq c_2\tfrac{16}{3} |y-x|^2,
    \end{align*}
    where we have used \eqref{duvf} in the last step.
    So we verified \eqref{Whitney_1} and \eqref{Whitney_2} with $c_3:=\max\{\frac{16}{3}c_2,c_3'\}$. 
    
    In the remaining case (Case IV) we obtain a corresponding $c_4$ similarly as in Case III. If, e.g., $x\in \overline{L}_\lambda$, $y\in C$ and $y>b_\lambda$ we use estimates from Case I for the points $y, b_\lambda$ and estimates from Case II for the points $b_\lambda$ and $x$. So we can put $c:=\max\{c_1,c_2,c_3,c_4\}$.

    We can now apply Fact~\ref{Wh} (Whitney's $C^{1,1}$ extension theorem) and find a $C^{1,1}$ mapping $F:\R^d\to\R^d$ such that $F(x)=f(x)$ and $DF(x)=\varphi(x)$ whenever $x\in A$. 
    Since $DF(\rho)=\mu_\rho$ is bijective (see \eqref{mu_bounded}), we can apply Fact~\ref{locdif} to $F$ and find an $\ep>0$ such that the restriction of $F$ to $B(\rho,\ep)$ is a $C^{1,1}$-diffeomorphism. We have $F(x)=x$ for all $x\in[0,2\rho]$ by \eqref{Pf}, in particular, $F(\rho)=\rho$.

Choose $0<\delta<\min\{\rho,\ep\}$ and define
$$A_1:=A\cap\overline{B}(\rho,\delta)\quad \text{ and }\quad  \tilde{A}:=F(A_1).$$
Using Lemma~\ref{skew-prop}~(ii) we obtain that $\reach A_1>0$, $\dim A_1\leq 2$ and $A_1$ is skewered on $[\rho-\delta,\rho+\delta]$. Applying Lemma~\ref{reach_C11} we get also $\reach \tilde{A}>0$, and clearly we have $\dim \tilde{A}\leq 2$. Equations \eqref{Pf} together with $F|_A=f$ yield $[\rho-\delta,\rho+\delta]\subset \tilde{A}\subset\pi^{-1}([\rho-\delta,\rho+\delta])$, and, with Lemma~\ref{tkdif}~(iii) also
\begin{equation} \label{T1again}
T_1(\tilde{A})=T_1(F(A_1))=F(T_1(A_1))=T_1(A_1)\subset [\rho-\delta,\rho+\delta];
\end{equation}
thus $\tilde{A}$ is skewered on $[\rho-\delta,\rho+\delta]$.
We also have $\rho\in T_1(\tilde{A})$ by \eqref{diam<} and \eqref{T1again}.

Further, for any $\lambda\in\Lambda$, the leave $A_\lambda$ fulfills $F(A_\lambda)\subset V_\lambda$, where $V_\lambda$ is the image of the linear mapping $\mu_{a_\lambda}\circ\pi_{T_\lambda}$ (this follows directly from the fact $f=F|_A$ and from the definition of $f$). We have $V_\lambda\in G(d,2)$  by \eqref{mu_bounded}.
Note that $T_1(A_1)=T_1(A)\cap[\rho-\delta,\rho+\delta]$ by Lemma~\ref{neprib}, hence
each leave of $A_1$ is a subset of a leave of $A$. Also, each leave of $\tilde{A}$ is the $F$-image of a leave of $A_1$ by \eqref{T1again} and \eqref{Pf}. Thus we get that $\tilde{A}$ has planar leaves. 

The proof is finished by setting $\tilde{S}:=\tilde{A}-\rho$ and $[\tilde{c},\tilde{d}]:=[-\delta,\delta]$, since $(S,0)\approx (A,\rho)\approx (\tilde{A},\rho)\approx (\tilde{S},0)$.
\end{proof}

\subsection{Proof of Proposition~\ref{T3}} 
In this subsection we will work with multirotations (see Definition~\ref{multrot}). We will use the easy fact that if $\rho:
\R^d\to\R^d$ is a multirotation than it is a bijection and $\rho^{-1}$ is again a multirotation associated with the same set as $\rho$.

The proof of Proposition~\ref{T3} is based on two lemmas.

\begin{lemma}\label{prorot} 
Let $A\subset\R^d$ ($d\geq  3$) be a set with positive reach skewered on a segment $[p,q]\subset \R$ and let $r>0$ be such that  $0< r < \reach A$ and 
\begin{enumerate}
    \item[{\rm (a)}] $\diam A < r$, or
    \item[{\rm (b)}] $A\subset \R^2 \subset \R^d$ is a $B$-set.
\end{enumerate}	
Let $\rho$ be a multirotation associated with $T_1(A)$.
Then 
\begin{enumerate}
\item[{\rm (i)}]
the mapping  $\rho|_A$ is bi-Lipschitz,
\item[{\rm (ii)}]
 $\rho(A)$ is compact and $\reach \rho(A) > 0$,
\item[{\rm (iii)}] $T_1(\rho(A))=T_1(A)$ and $\rho(A)$ is skewered on $[p,q]$.
\end{enumerate}  
\end{lemma}

\begin{proof}
We will treat both cases (a) and (b) together, using  that there exist $E>0$ and $F>0$
 such that in both cases, for each $z \in T_1(A)$ and $x\in A$,
\begin{equation}\label{EF}
|x-z| \leq E |\pi(x)-z|\quad \text{and}\quad |x-\pi(x)|\leq F|\pi(x)-z|^2. 
\end{equation}
To prove this claim, fix arbitrary $z \in T_1(A)$ and $x\in A$. First observe that Corollary~\ref{uhlovy} and Proposition~\ref{fedtan} imply that
\eqref{EF} holds with $E=E_1:= 2/\sqrt 3$ and $F=F_1:= 2/3r$ whenever $|z-x|<r$, and so always in case (a).

In case (b)  (in which  $[c,d]=[-r,r]$ or $[c,d]=[0,r]$) recall that functions $\vf$, $\psi$ from
Definition \ref{B-sets} ((i) or (ii)) are $L$-Lipchitz for some $L>0$. Since $z\in T_1(A)$ we have $\vf(z) = \psi(z)=0$ by Remark~\ref{Rem_listky}~(b) and clearly
  $|x-\pi(x)| \leq \max\{\vf(\pi(x)), -\psi(\pi(x))\}$. Thus we have
	$|x-\pi(x)|\leq L|\pi(x)-z|$ and consequently $|x-z| \leq  \sqrt{1+L^2} |\pi(x)-z|$.
	 So, if  $|\pi(x)-z| < r/\sqrt{1+L^2} $, then $|x-z|<r$ and so \eqref{EF} holds with  $E=E_2:=\sqrt{1+L^2}$
	 and $F=F_1$.  If  $|\pi(x)-z| \geq r/\sqrt{1+L^2} $, then 
$$  |x-\pi(x)|\leq L|\pi(x)-z| \leq L|\pi(x)-z|\cdot \frac{|\pi(x)-z|\sqrt{1+L^2}}{r}=: F_2|\pi(x)-z|^2 .$$
So, to prove our claim  we can choose $E= \max(E_1,E_2)$ and $F= \max(F_1,F_2)$.

To prove (i), consider two different points  $w_1,w_2 \in A$ and denote $\tilde{w}_i:=\rho(w_i)$, $i=1,2$. We can suppose $\pi(w_1)\leq\pi(w_2)$.
Let $\Lambda'$ be the family of all components $\lambda$ of $\R\setminus T_1(A)$ such that $\lambda\cap[p,q]\neq\emptyset$. Recall that $A_\lambda:=A\cap\pi^{-1}(\lambda)\neq\emptyset$, $\lambda\in\Lambda'$, are the leaves of $A$.

    First suppose that there exists  $\lambda \in \Lambda'$ such that   
    $w_1, w_2 \in \overline{A}_{\lambda}$. Using Lemma~\ref{skew-prop}~(v) we see that  $\rho|_{\overline{A}_{\lambda}}$ is an isometry on  $\overline{A}_{\lambda}$, hence we have  $|\tilde{w}_1-\tilde{w}_2| = |w_1- w_2|$.
		
    Second, if such  $\lambda \in \Lambda'$ does not exist, we can choose $z\in T_1(A)$ such that $\pi(w_1) \leq z \leq  \pi(w_2)$. Obviously, $\pi(\tilde{w}_i)=\pi(w_i)$ and $|\tilde w_i-\pi(\tilde{w}_i)|=|w_i-\pi(w_i)|$, which implies  $|w_i-z|= |\tilde w_i-z|$, $i=1,2$. 
    Using  \eqref{EF} with $x= w_i,\ i=1,2$, we obtain
	  $|w_i-z| \leq E|\pi(w_i)-z|$, $i=1,2$, which implies
    \begin{align*}
        |\pi(w_2)-\pi(w_1)|&\leq |w_1-w_2|\leq |w_1-z| +|w_2-z|\\ &\leq E (|\pi(w_1)-x| + |\pi(w_2)-x|) 
        = E |\pi(w_2)-\pi(w_1)|.
    \end{align*}
    Consequently 
    \begin{align*}
        |\pi(w_2)-\pi(w_1)|&\leq |\tilde{w}_1-\tilde{w}_2|\leq |\tilde{w}_1-z| +|\tilde{w}_2-z|\\ &=|w_1-z|+|w_2-z| 
        \leq E |\pi(w_2)-\pi(w_1)|,
    \end{align*}
    and the inequalities $ |\tilde w_1- \tilde w_2| \leq E|w_1-w_2|$, $ |w_1-w_2| \leq E|\tilde w_1- \tilde w_2| $ follow. Therefore $\rho|_A$ is bi-Lipschitz.
		
    To prove (ii), first observe that since $A$ is compact, $\rho(A)$ has the same property by (i).
    To prove  $\reach \rho(A) > 0$, we first observe that by Lemma~\ref{skew-prop}~(vi), for each $\lambda \in \Lambda'$, $\reach \overline{A}_\lambda>r$. Let $R_\lambda$ be the $2$-rotation agreeing with $\rho$ on $\pi^{-1}(\lambda)$, $\lambda \in \Lambda'$. Then for each $\lambda\in\Lambda'$, we have $\rho(\overline{A}_\lambda)=R_\lambda(\overline{A}_\lambda)$ by Lemma~\ref{skew-prop}~(v) and, hence,
    $\reach\rho(\overline{A}_\lambda)=\reach \overline{A}_\lambda>r$. (The last equality follows from the easy fact that isometries in $\R^d$ preserve the reach.)
    Now we will infer from Lemma~\ref{str} that
    $\reach \rho(A)\geq r':=\min\{\tfrac r2,\tfrac 1{8F}\}.$
    To this end, consider two arbitrary points $x,y\in \rho(A)$ with $|x-y|<2r'$; we can suppose that $\pi(x)\leq\pi(y)$. We will distinguish two cases.

    a)  There exists  $\lambda \in \Lambda'$ such that  $\pi(x),\pi(y)\in  \overline \lambda$. Since  $x,y \in \rho (\overline{A}_{\lambda})$, $\reach\overline{A}_\lambda>r$ and $r'\leq r/2$, Lemma~\ref{str} and Remark~\ref{ostr} give that
    \begin{equation}\label{exga*}
	 \left| \frac{x+y}{2} -s\right| \leq \frac{|x-y|^2}{4r}\leq \frac{|x-y|^2}{8r'}\ \text{ for some }
	\  s \in \rho(\overline{A}_{\lambda}) \subset \rho(A).
    \end{equation}
	
    b) If the case a) does not hold, there exists  $z\in T_1(A)$ such that
	 $\pi(x)\leq z \leq \pi(y)$. Since $\tilde{x}:=\rho^{-1} (x) \in A$, by  \eqref{EF} we obtain
	 $|x-\pi(x)| =|\tilde{x}-\pi(\tilde{x})|\leq F |\pi(\tilde{x})-z|^2 = F |\pi(x)-z|^2$. Quite similarly
	 we obtain  $|y-\pi(y)|  \leq  F |\pi(y)-z|^2$. 
    Now set  $s:= \frac{\pi(x)+\pi(y)}{2}$. Then   $s=\rho(s)\in \rho(A)$ and 
    \begin{align}
	 \left| \frac{x+y}{2} -s\right|&=\left|\frac{(x-\pi(x))+(y-\pi(y))}{2}\right| 
     \leq \frac{ |x-\pi(x)| + |y-\pi(y)|}{2} \nonumber\\
     &\leq \tfrac{F}{2} |\pi(x)-z|^2 + \tfrac{F}{2} |\pi(y)-z|^2 \leq F|\pi(x)-\pi(y)|^2\nonumber\\
     &\leq F|x-y|^2 \leq\frac{|x-y|^2}{8r'}.   \label{nexga*}
    \end{align}
Using \eqref{exga*}, \eqref{nexga*}, Lemma~\ref{str} and Remark~\ref{ostr}, we easily obtain that	$\reach \rho(A)\geq r'>0$.

It remains to verify (iii). Since $A$ is skewered on  $[p,q]$, the definition of  mutirotation
 easily gives  $[p,q]\subset\rho(A)  \subset \pi^{-1}([p,q])$. So it is sufficient to prove   
$T_1(A) = T_1(\rho(A))$. To this end, let $z^*=\rho(z) \in \rho(A)$ be given.  If $z^* \notin T_1(A)$
 then  $z\notin T_1(A)$. Observe that (see Lemma \ref{skew-prop} (iv)) $z \in A_{\lambda}$ for some
 $\lambda \in \Lambda'$ and  consider a  $2$-rotation  $R_\lambda$ as above.
 Applying  Lemma~\ref{tkdif} (ii) with ($\Phi:=R_\lambda|_{\pi^{-1}(\lambda)}$  and  $A:=A_{\lambda}$),
 we easily obtain that  $\dim \wtilde \Tan(A,z)= \dim \wtilde \Tan(\rho(A),z^*)$ and therefore
 $z^*  \notin T_1(\rho(A))$. Further, if $z^* \in T_1(A)$, then clearly $z=z^* \in T_1(A)$. So, if
 $x^*\in\rho(A)$ then, using \eqref{EF} for $z$ and $x:=\rho^{-1}(x^*)$, and the equalities 
$\pi(x)=\pi(x^*)$ and $|x-\pi(x)|=|x^*-\pi(x^*)|$, we obtain 
$$|x^*-\pi(x^*)|= |x-\pi(x)| \leq  F|\pi(x)-z|^2 = F|\pi(x^*)-z|^2. $$
This easily implies that $\Tan(\rho(A),z)\subset\R$, and, since $z\in [p,q]\subset\rho(A)$, we get $z^*=z\in T_1(\rho(A))$. Thus $T_1(A) = T_1(\rho(A))$ follows.
\end{proof}

In $\R^2$, skewered sets with positive reach are closely connected to $\B$-sets.

\begin{lemma}  \label{skew:B}
    Let $M\subset\R^2$ be a set with positive reach skewered on a segment $[c,d]\subset\R$, $c<0<d$, and let $0\in T_1(M)$. Then there exists a $\B_+$-set $B\subset\R^2$ such that $M\sim_0 B$.
\end{lemma}

\begin{proof}
    Since $M\subset\R^2$, $\reach(M,0)>0$ and $0\in T_1^+(M)$,  we can apply \cite[Theorem~6.4]{RZ17} and get that $M$ is of type $T_2$ at $0$ which means, by definition, that there exists a linear isometry $G:\R^2\to\R^2$ such that $G(M\cap B(0,r))$ is a $\tilde{T}_r^2$-set  for some $r>0$ (see \cite[Definition~6.3]{RZ17}). 
    Comparing the tangent cone $\Tan(M,0)=W$ with that of a $\tilde{T}_r^2$-set at $0$, we see that the linear isometry $G$  maps $e_1$ onto $\pm e_1$. Thus, the matrix of $G$ is diagonal with $\pm 1$ on the main diagonal, i.e., $G=I$ (identity), or $G:(x,y)\mapsto(-x,y)$, or $G:(x,y)\mapsto(x,-y)$, or $G:(x,y)\mapsto(-x,-y)$.
    
    As $M$ is skewered on $[c,d]$, it contains a neighbourhood of $0$ in $W$, and hence the $\tilde{T}_r^2$-set $G(M\cap B(0,r))$ has the same property.
    It is easy to see that a $\tilde{T}_r^2$-set containing a neighbourhood of the origin in $W$ locally agrees at $0$ with a $B_+$-set. Consequently, there exists a $\B_+$-set $B'\subset\R^2$ such that $G(M)\sim_0 B'$.  Due to the special form of the isometry $G$, and using Remark~\ref{rem_semicon}, we obtain that the preimage $B:=G^{-1}(B')$ is again a $B_+$-set. Since clearly $M\sim_0 B$, the proof is complete.
\end{proof}

\begin{proof}[Proof of Proposition~\ref{T3}]  \hskip 2mm
Choose $\omega,r$ with $0< 2\omega<r<\reach S$ and denote $S^*:=S\cap\overline{B}(0,\omega)$. 
Due to Lemma~\ref{skew-prop}~(ii), 
$\reach S^*>r$, $\dim S^*\leq 2$, $S^*$ skewered on $[c^*,d^*]:=[c,d]\cap[-\omega,\omega]$, and $0\in T_1(S^*)$.
Using also Lemma~\ref{neprib} we see that each leave of $S^*$ is a subset of a leave of $S$, hence $S^*$ has planar leaves as well. 
Let $\tilde{\rho}:\R^d\to\R^d$ be the multirotation associated with $T_1(S^*)$ (see Definition~\ref{multrot}) defined as follows (recall Notation~\ref{notation}): If $\lambda$ is a component of $W\setminus T_1(S^*)$ with $\lambda\cap [c^*,d^*]\neq\emptyset$ and $V_\lambda\in G(d,2)$ is the plane containing the leave $S^*_\lambda$, choose a unit vector $v_\lambda\in V_\lambda\cap W^\perp$ and define the $2$-rotation $\tilde{R}_\lambda$ as identity if $v_\lambda=\pm e_2$ and as the $2$-rotation mapping $v_\lambda$ to $e_2$ and being identity on $\spa\{v_\lambda,e_2\}^\perp$ otherwise. If $\lambda\cap[c^*,d^*]=\emptyset$ we set $\tilde{R}_\lambda:=\id$. The multirotation $\tilde{\rho}$ is then determined by the family $(\tilde{R}_\lambda)$, where $\lambda$ are the components of $W\setminus T_1(S^*)$.

Since $\reach S^*>r$ and $\diam S^*<r$, we can apply Lemma~\ref{skew-prop}~(iv) and Lemma~\ref{prorot} and obtain that $B':=\tilde{\rho}(S^*)\subset\spa\{e_1,e_2\}=\R^2\subset\R^d$ has positive reach (in $\R^d$), $B'$ is skewered on $[c^*,d^*]$, $T_1(B')=T_1(S^*)$ and clearly $\Tan(B',0)=W$.

Note that $B'$ has positive reach also as a subset of $\R^2$ (see Remark~\ref{rem_reach}). By Lemma~\ref{skew:B}, there exists a $\B_+$-set $B\subset\R^2\subset\R^d$ and a number $0<\tau<\min\{\reach B,\reach B'\}$ such that
\begin{equation}  \label{BB'}
B'\cap \overline{B}(0,\tau)=B\cap \overline{B}(0,\tau).
\end{equation}
Since $0\in T_1(B')$, we have $|x|=|\tilde{\rho}^{-1}(x)|$ for all $x\in\R^2$.
Thus also
$$S^*\cap \overline{B}(0,\tau)=\tilde{\rho}^{-1}(B)\cap \overline{B}(0,\tau).$$
Note that $\tilde{\rho}^{-1}$ is a multirotation associated with $T_1(B')=T_1(S^*)$.
Since $T_1(B')\cap\overline{B}(0,\tau)=T_1(B)\cap\overline{B}(0,\tau)$ by \eqref{BB'} and Lemma~\ref{neprib} (applied to both $B$ and $B'$), there exists a multirotation $\rho$ associated with $T_1(B)$ such that $\rho|_{\pi^{-1}([-\tau,\tau])}=\tilde{\rho}^{-1}|_{\pi^{-1}([-\tau,\tau])}$ (see Lemma~\ref{mrot}) and consequently 
$S^*\cap B(0,\tau)=\rho(B)\cap B(0,\tau)$. Hence
$S\sim_0 S^*\sim_0\rho(B)$ and the assertion follows.
\end{proof}

\subsection{Proofs of Theorem~\ref{T_main} and Corollary~\ref{bilip}}  \label{subs_last}
\begin{proof}[Proof of Theorem~\ref{T_main}]  \hskip 2mm
If $a\in T_1^+(A)$, the first part of the theorem (which is equivalent to Proposition~\ref{T}, see Subsection~\ref{ss51}) follows from Propositions~\ref{T1}, \ref{T2} and \ref{T3}. So it is sufficient to prove Proposition~\ref{T} in the case $a\in T_1^-(A)$.

In this case, choose $0<r<\reach A$ and let $u$ be the unit vector such that $\Tan(A,a)=\{tu:\, t\geq 0\}$. Denoting $S:=\{a+tu:\, -\frac r4\leq t< 0\}$ and $A^*:=A\cup S$, we have 
$\reach A^* \geq  r/4$ by \cite[Lemma 3.8]{RZ17}, and clearly $a\in T_1^+(A^*)$.
We apply the already proved assertion and find a $\B_+$-set $B^*\subset\R^2\subset\R^d$ and a multirotation $\rho^*$ associated with $T_1(B^*)$ such that $(A^*,a)\approx(\rho^*(B^*),0)$. 
By Remark~\ref{oekv1}~(i) there exist a $C^{1,1}$-diffeomorphism $\Phi$ defined on a neighbourhood $U$ of $a$ and $\ep>0$ such that 
\begin{equation} \label{AB}
\Phi(a)=0 \ \text{ and }\ \Phi(A^*\cap U)\cap B(0,\ep)=\rho^*(B^*)\cap B(0,\ep).
\end{equation}
Further, there exists $0<\delta<r/4$ such that
\begin{equation} \label{S_delta}
    B(a,\delta)\subset U\text{ and }\Phi(B(a,\delta))\subset B(0,\ep).
\end{equation}
Denote $S_\delta:=S\cap B(a,\delta)=\{a+tu:\, -\delta< t< 0\}$.
Since $S\subset T_1(A^*)$, we obtain
$$\Phi(S_\delta)\subset \Phi(T_1(A^*)\cap U)\cap B(0,\ep) \subset T_1(\rho^*(B^*))$$
using Lemma~\ref{tkdif}~(ii) (applied to $A:=A^*\cap U$) and \eqref{AB}.
We have further $T_1(\rho^*(B^*))=T_1(B^*)\subset W$ by Lemma~\ref{prorot}~(iii) and Remark~\ref{Rem_listky}~(b), hence we get $\Phi(S_\delta)\subset W$. Since  $\Phi(S_\delta)$ is homeomorphic to $(0,1)$ and $0\in\overline{\Phi(S_\delta)}\setminus \Phi(S_\delta)$, we have (a) $\Phi(S_\delta)=(-\ep',0)\subset W$ or (b) $\Phi(S_\delta)=(0,\ep')\subset W$ for some $0<\ep'\leq \ep$.
We claim that we can choose $B^*,\rho^*,\Phi,U,\ep$ and $\delta$ as above satisfying \eqref{AB} and \eqref{S_delta} such that the case (a) occurs. Indeed, if $B^*,\rho^*,\Phi,U,\ep$ and $\delta$ satisfy \eqref{AB}, \eqref{S_delta} and (b), consider the reflection $G:(x,z)\mapsto(-x,z)$ ($(x,z)\in\R\times\R^{d-1}$) and set $\tilde{B}^*:=G(B^*)$, $\tilde{\rho}^*:=G\circ\rho^*\circ G$, $\tilde{\Phi}:=G\circ\Phi$, $\tilde{U}:=U$, $\tilde{\ep}:=\ep$ and $\tilde{\delta}:=\delta$. Then both \eqref{AB} and \eqref{S_delta} remain true (with $B^*:=\tilde{B}^*$, $\dots$, $\delta:=\tilde{\delta}$), and it is easy to see (using Remark~\ref{rem_semicon}) that $\tilde{B}^*$ is again a $\B_+$-set and that $\tilde{\rho}^*$ is a multirotation associated with
 $T_1(\tilde{B}^*)$. Moreover, $\tilde{\Phi}(S_{\tilde{\delta}})=-\Phi(S_\delta)=(-\ep',0)$, i.e., case (a) occurs.

We further work with the so chosen $B^*,\rho^*,\Phi,U,\ep$ and $\delta$.
If $B^*$ has the form $B^*=\{(x,y):\, x\in[-v,v],\, \psi(x)\leq y\leq\varphi(x)\}$ for some $v>0$, we get $\ep'\leq v$ and $\psi(x)=\varphi(x)=0$ whenever $x\in(-\ep',0)$ (otherwise, $x\in \Phi(S_\delta)$ would not belong to $T_1(B^*)=T_1(\rho^* (B^*))$, see Lemma~\ref{prorot}~(iii) and Remark~\ref{Rem_listky}~(b)).
Thus we have
$$\Phi(S_\delta)=(-\ep',0)=\rho^*(B^*)\cap\{(x,z):\, x<0\}\cap B(0,\ep').$$
It follows that, denoting $U':=B(a,\delta)$, 
\begin{align*}
    \Phi(A\cap U')\cap B(0,\ep')&=\Phi((A^*\setminus S)\cap U')\cap B(0,\ep')\\
    &= (\Phi(A^*\cap U')\setminus \Phi(S_\delta))\cap B(0,\ep')\\
    &=\rho^*(B^*)\cap\{(x,z):\, x\geq 0\}\cap B(0,\ep')\\
    &= \rho^*(B^*\cap\{(x,z):\, x\geq 0\})\cap B(0,\ep').
\end{align*}
Note that $B:=B^*\cap\{(x,y):\, x\geq 0\}\subset\R^2\subset\R^d$ is a $\B_-$-set.
Using Lemma~\ref{mrot} (applied with $\rho:=\rho^*$, $C:=T_1(B^*)$, $C':=T_1(B)$ and $[p,q]:=[0,v]$), we can find a multirotation $\rho$ associated with $T_1(B)$ and such that $\rho(B)=\rho^*(B)$. Then
$$\Phi(A\cap U')\cap B(0,\ep')=\rho(B)\cap B(0,\ep'),$$
which proves the assertion of Proposition~\ref{T}.

To prove the second part of Theorem~\ref{T_main}, assume that $B\subset\R^2\subset\R^d$ is a $\B$-set, $\rho$ a multirotation associated with $T_1(B)$, $U\supset\rho(B)$ is open and $\Phi:U\to\R^d$ a $C^{1,1}$-diffeomorphism with $\Phi(0)=a$. Then $\rho(B)$ is compact and has positive reach by Lemma~\ref{prorot}~(ii). Setting $A:=\Phi(\rho(B))$, we have $\reach A>0$ by Lemma~\ref{reach_C11}. We have $0\in T_1(\rho(B))$ by Lemma~\ref{prorot}~(iii) and, hence, $a\in T_1(A)$ by Lemma~\ref{tkdif}~(iii). 
Further, $\dim\rho(B)\leq 2$ since $\rho|_B$ is a homeomorpism by Lemma~\ref{prorot}~(i) and, hence, also $\dim A\leq 2$, since $A$ is homeomorpic to $\rho(B)$.
So the second part of Theorem~\ref{T_main} is proved.
\end{proof}

\begin{proof}[Proof of Corollary~\ref{bilip}]  \hskip 2mm
    Let $A\subset\R^d$ be a two-dimensional set with positive reach and $a\in A$. 
    If $a\in T_0(A)$ then $a$ is an isolated point in $A$ and the assertion is obvious.
    If $a\in T_1(A)$ then we apply Theorem~\ref{T_main} and obtain a $\B$-set $B\subset\R^2$, a multirotation $\rho$ associated with $T_1(B)$, an open set  $U\supset\rho(B)$ and a $C^{1,1}$-diffeomorphism $\Phi:U\to\R^d$ such that $\Phi(\rho(B))$ is a neighbourhood of $a$ in $A$. 
    The set $\rho(B)$ is compact by Lemma~\ref{prorot}~(ii), hence $\Phi(\rho(B))$ is compact as well. Since both $\Phi$ and $\Phi^{-1}$ are locally Lipschitz (see Remark~\ref{stability}~(i)), their restrictions to compact sets $\rho(B)$ and $\Phi(\rho(B))$ (respectively) are Lipschitz by \cite[Proposition~A48]{Lee}, i.e., $\Phi$ is bi-Lipschitz on $\rho(B)$.
    Since $\rho$ is bi-Lipschitz on $B$ by Lemma~\ref{prorot}~(i), we can choose $\Psi:=\Phi\circ\rho|_B$ and the assertion follows.

    Similarly, if $a\in T_2(A)$ then we use Theorem~L and obtain corresponding $K$, $U$ and $\Phi$. As above, we infer that $\Phi$ is bi-Lipschitz on $K$ and we obtain the assertion with $B:=K$ and $\Psi:=\Phi|_K$. 
\end{proof}
\medskip 

\section{An alternative proof of Theorem~L}

In this section we present, using Theorem~\ref{tdmnapl2}, an alternative  proof of Lytchak's Theorem~L.
Although the basic ideas of our approach come from \cite{Ly24}, we present our proof
since it is more elementary, as it does not need results from the theory of length spaces.
Moreover, our proof of Proposition~\ref{PD} (which is similar to the proof from  
\cite[Subsection 5.2]{Ly24}) uses Lemma~\ref{Fu_local} (easily inferred from \cite{Fu85}) whereas the proof from \cite{Ly24} uses a (closely related) result on ``$O_{\delta}$ sets'' from paper \cite{Re56} which is accessible in Russian only.

Assertion (ii) of the following proposition clearly implies the assertion of Theorem~L for $k=d\geq 2$. It is proved via (i), which is of some independent interest (and comes also from \cite{Ly24}, where it is used implicitly).

\begin{proposition} \label{PD}	
Let $A\subset\R^d$ ($d\geq 1$) have positive reach and $a\in T_d(A)\cap\partial A$. Then the following statements hold.
\begin{enumerate}
    \item[{\rm (i)}] If $d\geq 2$ then there exist $\ep>0$, $u\in S^{d-1}$ and a Lipschitz semiconcave function $f:u^\perp\to\R$ such that
    \begin{equation}  \label{loc_sc}
        A\cap B(a,\ep)=\hyp_u f\cap B(a,\ep).
    \end{equation}
    \item[{\rm (ii)}] There exists a convex body $K\subset\R^d$ and a surjective $C^{1,1}$-diffe\-omor\-phism $\Phi:\R^d\to\R^d$  such that
	\begin{equation}\label{ok2}
	\Phi(K) \ \text{ is a neighbourhood of }\ a\  \text{ in }\  A.
	\end{equation}
\end{enumerate}
\end{proposition}

\begin{proof}
    (i). Assume that $d\geq 2$ and choose $0<r<\reach A$. Since $\Tan(A,a)$ is a full-dimensional tangent cone, we can choose a unit vector $v_0\in\INt\Tan(A,a)$ and, using Lemma~\ref{P_PR}~(vii) and Lemma~\ref{conv_val}, we easily obtain $\omega>0$ such that $v\in\INt\Tan(A,z)$ whenever $z\in A\cap B(a,\omega)$ and $\angle(v,v_0)<\omega$. We can assume without loss of generality that $a=0$ and $v_0=-e_d$. Further, if $z\in\partial A$ and $v\in\INt\Tan(A,z)$ then $-v \not\in\Tan(A,z)$ (since otherwise, the convex cone $\Tan(A,z)$ would contain a neighbourhood of $0$ and so it would be the whole $\R^d$, which is impossible, see Lemma~\ref{P_PR}~(iii)). Thus we get for any $z\in \partial A\cap B(0,\omega)$ that
    \begin{align}
      &\text{if }\angle(v,-e_d)<\omega\text{ then }v\in\INt\Tan(A,z),\quad\text{and}\\
      &\text{if }\angle(v,e_d)<\omega\text{ then }v\not\in\Tan(A,z). \label{***}
    \end{align}
    Denote $W:=e_d^\perp$. Applying \cite[Lemma~3.5]{RZ17} and the definition of tangent vectors, we find that there exist $\tau>0$ and $L>0$ such that for any $w\in W\cap B(0,\tau)$,
    \begin{equation}  \label{kuzely}
        w-L|w|e_d\in A\ \text{ and } \ w+L|w|e_d\not\in A.
    \end{equation}
    Further, choosing $0<\eta<\min\{\omega,r\omega/6\}$, we get
    \begin{equation}  \label{sklon}
        \text{if }z_1\neq z_2\in\partial A\cap B(0,\eta)\text{ then }\angle(z_2-z_1,e_d)>\tfrac\omega 3.
    \end{equation}
    Indeed, applying Corollary~\ref{uhlovy} with $a:=z_1$ and $b:=z_2$ we get that there exists $0\neq v\in\Tan(A,z_1)$ such that $\angle(z_2-z_1,v)<\frac\omega 3$, and since $\angle(v,e_d)\geq\omega$ by \eqref{***}, we get $\angle(z_2-z_1,e_d)>\tfrac\omega 3$.
    Take further $0<\delta<\tau$ so that $\delta+L\delta<\eta$. Then, \eqref{kuzely} and \eqref{sklon} imply that for any $w\in W\cap B(0,\delta)$ there exists exactly one $t=:\varphi(w)\in\R$ such that $w+te_d\in\partial A\cap B(0,\eta)$. Moreover,
    \begin{equation}\label{ahv}
        A\cap B(0,\delta)=\hyp\varphi\cap B(0,\delta),
    \end{equation}
	$\vf(0) = 0$, and \eqref{sklon} easily implies that $\vf$ is Lipschitz. 
    Applying Lemma~\ref{Fu_local} (with Lemma~\ref{local}) to a Lipschitz extension of $\vf$ to $W$ we obtain that there exists $\zeta>0$ such that $\varphi$ is semiconcave on $B:=B(0,\zeta)\cap W$. Taking for $f$ a Lipschitz semiconcave extension of $\varphi|_B$ to $W$  (which exists by \cite[Proposition~1.7]{Fu85} and Remark~\ref{RemFu}), we can by \eqref{ahv} choose  $\ep>0$ so small that \eqref{loc_sc} (with $u=-v_0$) holds. This completes the proof of (i).  

    (ii). Assume first that $d\geq 2$. Choose a concave function $g$  on $W$ and $c>0$ such that $f(w)=g(w)+c|w|^2,\ w\in W$. Then, using
	the standard identifications $W= \R^{d-1}$ and $\R^d= W \times \R$, we have
		 $\hyp f = \{(w,t):\ t \leq g(w)+c|w|^2\}$. Now, for $(w,t) \in \R^d$, set
		$$  \Phi(w,t):= (w, t+ c|w|^2)\ \ \text{and}\ \ \Psi(w,t):= (w, t- c|w|^2).$$
	Then, since  $\Psi= \Phi^{-1}$ and $w \mapsto |w|^2$ is $C^{1,1}$ smooth, we easily obtain that  
	$\Phi: \R^d \to \R^d$ is a bijective $C^{1,1}$ diffeomorphism with $\Phi(0)=0$. Moreover, clearly
	\begin{equation}\label{hyphyp}
		\Phi(\hyp g) = \hyp f.
	\end{equation}
	Choose $\delta>0$ such that  $\Phi(\overline B_{\delta}(0)) \subset B(0,\ep)$ and set
	$K:= \hyp g \cap \overline B_{\delta}(0)$. Then $K$ is a convex body and, using
	\eqref{hyphyp} and \eqref{loc_sc}, we obtain that $\Phi(K)$ is a neighbourhood of $a$ in $A$.

    Finally, consider the case $d=1$. From Lemma~\ref{ball} and Lemma~\ref{P_PR}~(v) we easily obtain that there exists $\ep>0$ such that $A\cap \overline{B}(a,\ep)=A\cap [a-\ep,a+\ep]$ equals either $K:=[a-\ep,a]$, or $K:=[a,a+\ep]$, and (ii) holds with $\Phi:=\id$.
\end{proof}

\begin{proof}[Proof of Theorem~L]  \hskip 2mm
If $k=d$, the assertion of Theorem~L follows from Proposition~\ref{PD}~(ii), since the case $a\in \interior A$ is
 trivial. 

Consider now the case $1\leq k \leq d-1$. Then
by Theorem \ref{tdmnapl2} (using also Definition~\ref{plochy}) we can choose  $\delta_1>0$ and  $k$-dimensional $C^{1,1}$ g-surface	 $\Gamma$ such that  $T_k(A) \cap B(a, \delta_1) \subset \Gamma$.
The assumptions together with Lemma~\ref{T_k}~(i), Lemma~\ref{P_PR}~(iv) and  Lemma~\ref{ball} imply that there exists $0< \delta_2<\delta_1$ such that 
	$$B:= \overline B(a, \delta_2)\cap A \subset T_k(A)   \ \ \text{and}\ \  \reach B>0,$$
	 and consequently
	\begin{equation}\label{BGa}
	B \subset T_k(A) \cap \overline B(a, \delta_2) \subset \Gamma.
	\end{equation}
	By definition, there exists $W \in G(d,k)$ and a $C^{1,1}$ mapping  $F: W \to W^{\perp}$ 
	 such that  $\Gamma = \{x+ F(x):\ x \in W\}$. Without any loss of generality we can suppose
	 that $a=0$ and $W= \R^k = \R^k \times \{0\}$, where we use the standard identifications
	 $\R^d= \R^k \times \R^{d-k}$, $\R^k = \R^k \times \{0\}$ and $W^{\perp}=  \{0\}\times \R^{d-k} = \R^{d-k}$.
	Using these conventions,  $\Gamma = \{(x,F(x)):\ x \in \R^k\}$.

	 Now we define the surjective $C^{1,1}$-diffeomorphism 
	$$   \eta: \R^d \to \R^d,\ \ \ \eta(x,y) = (x,y- F(x)).$$
	(Note that  $\eta^{-1}(x,y)= (x,y+F(x)), \ (x,y) \in \R^d$.)
	 Then clearly $C:=\eta(B) \subset  \R^k \times \{0\} = \R^k$ and $\eta(0)=0\in C$. By Lemma~\ref{reach_C11}  $C$  has positive reach in $\R^d$,
	 and so  also in $\R^k$ (see Remark~\ref{rem_reach}). 
	Moreover, by   Lemma    \ref{tkdif} (iii), we have  $\dim \Tan(C, 0) =k$.
	 Then by Proposition~\ref{PD} (ii) (the case $a \in \interior_{\R^k} C$ is trivial) there exists a convex body  $K_1\subset \R^k$
	  and a surjective  $C^{1,1}$-diffeomorphism  $\omega: \R^k \to \R^k$ 
		 such that $\omega(K_1)$ is a neighbourhood of $0$ in $C$. So, putting
		$$  \Omega(x,y) = (\omega(x), y) \in \R^k \times \R^{d-k} = \R^d$$
		and considering $K_1$ as a subset of $W=\R^k \times \{0\}$, we have that 
		  $\Omega: \R^d \to \R^d$ is a  surjective  $C^{1,1}$-diffeomorphism and  $\Omega(K_1)$ is a 
			neighbourhood of $0$ in $C$. Setting $\Psi:= \eta^{-1} \circ \Omega$, we obtain that  
			 $\Psi: \R^d \to \R^d$ is a  surjective  $C^1$-diffeomorphism and  $\Psi(K_1)$ is a neighbourhood of 
			$0$ in $A$. Moreover, by Remark  \ref{stability} (c) we have that both $\Psi$ and $\Psi^{-1}$
			 are locally $C^{1,1}$ smooth. Denote $\tilde a:= \Psi^{-1}(0)$ and observe that there exists
			 $\sigma>0$ such that $\Phi:= \Psi|_{B(\tilde a, \sigma)}$ is a $C^{1,1}$-diffeomorphism.
			 Now the choice  $U:= B(\tilde a, \sigma)$ and  $K:= K_1 \cap \overline B(\tilde a, \sigma/2)$
			 shows that the assertion of Theorem~L holds.
		\end{proof}

\subsection*{Acknowledgement}
    The authors are grateful to Alexander Lytchak for helpful comments on an earlier version of the manusript.

\end{document}